\documentclass{tran-l}

\usepackage{verbatim}
\usepackage{appendix}
\usepackage{amssymb}
\usepackage{amsmath}
\usepackage{xcolor}
\newtheorem{theorem}{Theorem}[section]
\newtheorem{lemma}[theorem]{Lemma}
\newtheorem{proposition}[theorem]{Proposition}

\newtheorem*{observation}{Observation}

\newtheorem{corollary}[theorem]{Corollary}
\theoremstyle{definition}
\newtheorem{definition}[theorem]{Definition}

\theoremstyle{remark}
\newtheorem{remark}[theorem]{Remark}
\numberwithin{equation}{section}
\newtheorem{question}{Question}[section]

\DeclareMathOperator{\diag}{diag}

\newif\ifshowrevisions
\showrevisionsfalse

\newcommand{\rev}[1]{\ifshowrevisions{\color{red}#1}\else#1\fi}

\newenvironment{revblock}{%
    \ifshowrevisions\color{red}\fi
}{}

\begin{document}

\title[HAUSDORFF DIMENSION OF ANOSOV SUBGROUPS' LIMIT SETS]{Hausdorff Dimension of Anosov Subgroups' Limit Sets with Special Self-Affine Complexity}

\author{Zhufeng Yao}

\thanks{The author was supported by the National University of Singapore and the Ministry of Education grant A-8004148-00-00.}



\date{}

\dedicatory{}

\begin{abstract}
    Let $\Gamma\subset \mathsf{PGL}(d,\mathbb{R})$ be an irreducible projective Anosov subgroup and let $\Lambda^1(\Gamma)$ be its projective limit set. Viewing $\Lambda^1(\Gamma)$ as an analogue of a self-affine set, we investigate the Hausdorff dimension of $\Lambda^1(\Gamma)$ under specific assumptions regarding its affine complexity:
    \begin{itemize}
        \item If $\Lambda^1(\Gamma)$ is of full Hausdorff dimension, then $d= 2$ and $\Gamma$ is a cocompact lattice.
        
        \item If $d = 3$ and $\Gamma$ is the image of a closed surface group under an irreducible Anosov representation, then $\Lambda^1(\Gamma)$ never has Hausdorff dimension $1$ unless the representation is Hitchin.

        \item If the limit set $\Lambda^1(\Gamma)$ exhibits \rev{a form of
        quasi-self-similarity inspired by Falconer's work}---which can be \rev{implied} by the ``regular distortion property'' of $\Gamma$---then the Hausdorff dimension of $\Lambda^1(\Gamma)$ equals the critical exponent of the first simple root. An application of this result is the computation of the Hausdorff dimension of the limit set for arbitrary $\Theta$-positive representations of convex cocompact Fuchsian groups.
    \end{itemize}
\end{abstract}

\maketitle
\setcounter{tocdepth}{1}

\tableofcontents

\section*{Declaration of Generative AI Usage}
During the preparation of this manuscript, the author utilized Gemini and Chatgpt to assist with writing, refine grammar and vocabulary, improve sentence flow, and correct typographical errors.

\section{Introduction}
Patterson~\cite{patterson1976limit} showed that for a convex cocompact Fuchsian group $\Gamma_0$ acting on the hyperbolic disk $\mathbb{D}$, the Hausdorff dimension of its limit set $\Lambda(\Gamma_0)\subset \partial \mathbb{D}$ equals the exponential growth rate of its orbit with respect to the hyperbolic distance, by constructing a canonical Hausdorff measure on $\Lambda(\Gamma_0)$. Sullivan~\cite{sullivanoriginal} generalized this construction---now referred to as the Patterson--Sullivan measure---to show the same equality for the Hausdorff dimension of limit sets of convex cocompact subgroups acting on $d$-dimensional hyperbolic spaces. Subsequently, these constructions have been generalized to study the asymptotic and fractal properties of discrete subgroups acting on various geometric spaces. 

In this paper, we focus on the setting of Anosov subgroups in $\mathsf{PGL}(d,\mathbb{R})$, a class of groups introduced by Labourie~\cite{labourie2005anosovflowssurfacegroups} and extended by Guichard--Wienhard~\cite{GW1} as a higher-rank generalization of convex cocompact Fuchsian groups (see also Kapovich--Leeb--Porti~\cite{kapovich2014morse, kapovich2017anosov, anosovmorseandbuilding} and Gu\'eritaud--Guichard--Kassel--Wienhard~\cite{Gu_ritaud_2017} for further developments). For references regarding Patterson--Sullivan theory in this setting, see Albuquerque~\cite{Albuquerque}, Quint~\cite{quintoriginal}, Sambarino~\cite{sambarinopattersonsullivan}, Dey--Kapovich~\cite{deykapovichps},  Burger--Landesberg--Lee--Oh~\cite{BurgerLandesbergLeeHee}, Lee--Oh~\cite{leeohps1, leeohps2}, Dey--Kim--Oh~\cite{heeohregularity}, and many others.

For the remainder of this paper, we restrict our attention to the group $\mathsf{PGL}(d,\mathbb{R})$. However, we emphasize that all our discussions apply equally to $\mathsf{PSL}(d,\mathbb{R})$ (and note that $\mathsf{PGL}(d,\mathbb{R}) \cong \mathsf{PSL}(d,\mathbb{R})$ when $d$ is odd). Fix an Euclidean metric on $\mathbb{R}^{d}$. For any $g\in \mathsf{PGL}(d,\mathbb{R})$  and $1\le i, j \le d$, the singular value ratio $\frac{\sigma_i}{\sigma_j}(g)$ is well-defined, independent of the choice of the representative of $g$ in $\mathsf{GL}(d,\mathbb{R})$. A finitely generated subgroup $\Gamma\subset \mathsf{PGL}(d,\mathbb{R})$  is called \emph{projective Anosov} if there exist $a>1$ and $A>0$ such that for any $\gamma\in \Gamma$,
\[
\frac{1}{a} d_{word}(\mathrm{id},\gamma)-A \le \log \frac{\sigma_1}{\sigma_2}(\gamma) \le a d_{word}(\mathrm{id},\gamma) + A.
\]
Here, the word metric $d_{word}$ depends on an implicit choice of a finite generating set which will not affect the definition of projective Anosov subgroups. The \emph{projective limit set} $\Lambda^1(\Gamma)\subset \mathbb{RP}^{d-1}$ is defined as \rev{the closure of the set of attracting fixed points of
infinite-order elements}. When $|\Lambda^1(\Gamma)|=\infty$, $\Gamma$ is called \emph{non-elementary}, in which case $\Lambda^1(\Gamma)$ is a perfect set (i.e., an uncountable closed set with no isolated points). It can be shown that any projective Anosov subgroup is Gromov 
hyperbolic (see Kapovich--Leeb--Porti~\cite[Theorem~1.5]{anosovmorseandbuilding}), and the limit set is equivariantly homeomorphic to the Gromov boundary $\partial \Gamma$.

As mentioned, a typical example is given by convex cocompact Fuchsian subgroups $\Gamma_0\subset \mathsf{PSL}(2,\mathbb{R})$, whose limit set can be identified with the set of accumulation points in the boundary $\partial \mathbb{D}$ of any orbit $\Gamma_0 b_0$, $b_0\in \mathbb{D}$. As a generalization, let $\Gamma_0$ be a closed surface group and let $\rho_d:\mathsf{PSL}(2,\mathbb{R})\to \mathsf{PSL}(d,\mathbb{R})$ be the irreducible representation. Then, for any representation $\rho \in \mathrm{Hom}(\Gamma_0, \mathsf{PSL}(d,\mathbb{R}))$ whose conjugacy class can be continuously deformed to the class of $\rho_d|_{\Gamma_0}$ in the representation variety $\mathrm{Hom}(\Gamma_0, \mathsf{PSL}(d,\mathbb{R}))/\mathsf{PGL}(d,\mathbb{R})$, the exterior wedge power $\wedge^{k} \rho(\Gamma_0)$ is projective Anosov for any $1\le k <d$ (see Labourie~\cite{labourie2005anosovflowssurfacegroups}). The connected component of $\mathrm{Hom}(\Gamma_0, \mathsf{PSL}(d,\mathbb{R}))/\mathsf{PGL}(d,\mathbb{R})$ containing the class of $\rho_d|_{\Gamma_0}$ is called the \emph{Hitchin component}, and representations whose conjugacy classes belong to this component are called \emph{Hitchin representations}. The Anosov property is intensively exploited in~\cite{labourie2005anosovflowssurfacegroups} to study the Hitchin component, which serves as a higher-rank analogue of Teichm\"uller space.

But in the higher-rank setting, even when the subgroup is Anosov, it is difficult to compute the Hausdorff dimension of its limit set. Unlike the rank-one case, the action of $\Gamma$ on $\Lambda^1(\Gamma)$ lacks conformality (which was essential in Patterson and Sullivan's work), and its dimension depends on multiple singular value gaps. There are numerous results under various assumptions; we provide a non-exhaustive list here:

\begin{enumerate}
    \item Glorieux--Monclair--Tholozan~\cite{glorieux2023hausdorffdimensionlimitsets} proved that for a projective Anosov subgroup, the Hausdorff dimension of the symmetric limit set with respect to the Riemannian metric is bounded above by the exponential growth rate of its first singular value gap, and bounded below by the exponential growth rate of a translation length function.
    \item Let $\Gamma$ be a finitely generated group and let $\rho_i:\Gamma\to \mathsf{SO}^{\circ}(n_i,1)$ for $i = 1,\dots,k$ be convex cocompact representations. Kim--Minsky--Oh~\cite{kim2023hausdorffdimensiondirectionallimit} showed that the Hausdorff dimension of the limit set of the diagonal group $(\prod_{i = 1}^{k}\rho_i)(\Gamma)$ equals the maximum of the exponential growth rates of the orbits under the action of $\rho_i(\Gamma)$.
    \item Dey--Kim--Oh~\cite{dey2025ahlforsregularitypattersonsullivanmeasures} proved that for $\theta$-Anosov Zariski-dense subgroups of real semisimple algebraic groups, the Hausdorff dimension of the $\theta$-limit set with respect to the symmetrized visual pre-metric equals the critical exponent of the symmetrized functional determining the visual pre-metric (see also Dey--Kapovich~\cite{Dey_2022}).
    \item Pozzetti--Sambarino--Wienhard~\cite{pozzetti2020conformalityrobustclassnonconformal} proved that if $\Gamma$ acts on $\Lambda^1(\Gamma)$ with a mild conformal property, then the Hausdorff dimension of its limit set equals the exponential growth rate of its first singular value gap. They provided a class of subgroups satisfying this mild conformal property, called $(1,1,2)$-hyperconvex subgroups, and it is known that wedges of the image of Hitchin representations fall into this class. As a result, the picture in the rank-one setting generalizes to the higher Teichm\"uller scenario. (We also note that Farr\'e--Pozzetti--Viaggi~\cite{farre_topological_2025} later showed that $(1,1,2)$-hyperconvex subgroups in $\mathsf{PSL}(d,\mathbb{C})$ are virtually isomorphic to Kleinian groups, indicating an essential connection to the rank-one phenomenon; see also Canary--Zhang--Zimmer~\cite{cazhzikleinian}).
    \item See Bishop--Jones~\cite{bishop1994hausdorffdimensionkleiniangroups}, Canary--Zhang--Zimmer~\cite{CaZhZi}, and many others for results dealing with more general classes of groups beyond Anosov subgroups.
\end{enumerate}

To address this non-conformality, we note that philosophically, the projective limit set of an Anosov subgroup shares common features with self-affine sets. Pozzetti--Sambarino--Wienhard~\cite{Liplim} thus employed the affinity exponent---a concept that first appeared in Kaplan--Yorke~\cite{affinityexponentkaplanyorke} and Douady--Oesterl\'e~\cite{affinityexponentdouadyoesterle}, and was later exploited by Falconer~\cite{falconerselfaffine} to study the fractal geometry of self-affine sets---to provide a potentially optimal estimate for the Hausdorff dimension of the limit set as follows:

\vspace{1mm}

Let $\Gamma\subset \mathsf{PGL}(d,\mathbb{R})$ be a discrete subgroup. We consider the piecewise Dirichlet series defined for $p\in \{1,\dots,d-2\}$ (or $p = d-1$) and $s \in [p-1, p]$ (or $s\in[d-2,\infty)$ when $p = d-1$) by
\[
\Phi_\Gamma^{\mathrm{Aff}}(s) := \sum_{\gamma \in \Gamma} \left( \frac{\sigma_2}{\sigma_1}(\gamma) \cdots \frac{\sigma_p}{\sigma_1}(\gamma) \right) \left( \frac{\sigma_{p+1}}{\sigma_1}(\gamma) \right)^{s-(p-1)}.
\]
This is called the \emph{affine series} of $\Gamma$.

The \emph{affinity exponent} of $\Gamma$, denoted by $\alpha(\Gamma)$, is defined as the critical exponent of this affine series:
\[
\alpha(\Gamma) := \inf \left\{ s \ge 0 \;\middle|\; \Phi_\Gamma^{\mathrm{Aff}}(s) < \infty \right\} = \sup \left\{ s \ge 0 \;\middle|\; \Phi_\Gamma^{\mathrm{Aff}}(s) = \infty \right\} \in [0, \infty].
\]

We have the following bound:

\begin{theorem}[{Pozzetti--Sambarino--Wienhard~\cite[Theorem~B]{Liplim}}]\label{pswtheomrem}
Let $\Gamma \subset \mathsf{PGL}(d,\mathbb{R})$ be a projective Anosov subgroup. Then the Hausdorff dimension of the projective limit set $\Lambda^1(\Gamma)$ is bounded above by the affinity exponent:
\[
\dim_H \Lambda^1(\Gamma) \le \alpha(\Gamma).
\]
\end{theorem}

It is still open to show whether $\dim_H \Lambda^1(\Gamma) = \alpha(\Gamma)$ holds in general. Li--Pan--Xu~\cite[Theorem~1.2]{lipanxu} confirmed this when $d = 3$ and $\Gamma$ is irreducible. They posed the following questions regarding the special dimensions:

\begin{question}\label{question:dim_less_than_2}
If $\Gamma\subset \mathsf{PGL}(3,\mathbb{R})$ is projective Anosov, is $\dim_H \Lambda^1(\Gamma)<2$ always true?
\end{question}

\begin{question}\label{question:hitchin_rigidity}
If $\Gamma\subset \mathsf{PGL}(3,\mathbb{R})$ is irreducible and can be written as $\rho(\Gamma_0)$, where $\Gamma_0$ is a closed surface group and $\rho$ is an Anosov representation, does $\dim_H \Lambda^1(\Gamma) = 1$ imply that $\rho$ is a Hitchin representation?
\end{question}

In this paper, we apply Patterson--Sullivan theory to give affirmative answers to these questions. The proofs essentially rely on geometric constructions arising from the special self-affine complexity at integer dimensions.

\begin{theorem}[Theorem~\ref{theoremoffulldimension}]\label{theoremoffulldimensioninintroduction}
If $\Gamma\subset \mathsf{PGL}(d,\mathbb{R})$ is a projective Anosov subgroup and the Hausdorff dimension of its projective limit set $\dim_H\Lambda^1(\Gamma)$ equals $d-1$, then $d = 2$ and $\Gamma$ is a cocompact lattice in $\mathsf{PGL}(2,\mathbb{R})$.
\end{theorem}

\begin{remark}
This improves Ledrappier--Lessa~\cite[Theorem~1.1]{ledrappier2023dimensiongapvariationalprinciple} when $P =\mathcal{Q} = \{1\}$ there.
\end{remark}

\begin{theorem}[Theorem~\ref{equivalenceofdim1}]\label{equivalenceofdim1inintro}
Let $\Gamma_0\subset \mathsf{PSL}(2,\mathbb{R})$ be a closed surface group and let $\rho:\Gamma_0\to \mathsf{PGL}(3,\mathbb{R})$ be an irreducible projective Anosov representation. Then $\dim_H \Lambda^1(\rho(\Gamma_0)) = 1$ if and only if $\rho$ is a Hitchin representation.
\end{theorem}

On the other hand, as we mentioned, the difficulty of computing the Hausdorff dimension of the limit set of \rev{general} Anosov subgroups essentially comes from the non-conformality of the group action on the limit set, which 'folds' the limit set to have multiple scale complexity in Euclidean balls (for instance, if $\Lambda^1(\Gamma)$ is homeomorphic to a circle, the intersection of $\Lambda^1(\Gamma)$ with a small Euclidean ball $B_r(x)$ centered at $x \in \Lambda^1(\Gamma)$ \rev{can} contain multiple arcs, and the number of these arcs \rev{can diverge} as $r \to 0$). We give a new conformal assumption on the subgroup (inspired by Pozzetti-Sambarino--Wienhard~\cite{pozzetti2020conformalityrobustclassnonconformal}, but the class of such subgroups is strictly larger), called the \emph{non-folding property}, to avoid such difficulty (see Definition~\ref{definitionofnonfolding}, and see Remark~\ref{listofexamples} for various examples satisfying this assumption). It then turns out that this non-folding property implies that the limit sets satisfy a quasi-self-similarity as in Falconer~\cite{falconerselfsimilar1}, which is applied by us to show:

\begin{theorem}[Theorem~\ref{theoremofnonfoldingdimension}]\label{nonfoldingdimensioninintroduction}
Let $\Gamma\subset \mathsf{PGL}(d,\mathbb{R})$ be a projective non-folding Anosov subgroup. Then 
\[
\dim_H \big(\Lambda^1(\Gamma)\big) =  \delta^{\alpha_1}(\Gamma),
\]
where $\delta^{\alpha_1}(\Gamma)$ is the critical exponent defined by the convergence abscissa of the associated Poincar\'e series:
\[
\mathcal{Q}_{\Gamma}^{\alpha_1}(s) = \sum_{\gamma \in \Gamma} \left(\frac{\sigma_2}{\sigma_1}(\gamma) \right)^{s}.
\]
\end{theorem}

Referring to Remark~\ref{listofexamples}, it is worth noting that the equality between the Hausdorff dimension of the limit set and the simple root critical exponent has already been established for most of the examples listed therein (see the remark for related references). However, we emphasize the novelty that this result also holds for all $\Theta$-positive representations and all the roots in $\Theta$, and it partially answers a question raised by Yao~\cite[Remark~1.9]{yaothetapositive}. We state it explicitly here using the notation from \cite[Section~1]{yaothetapositive}:

\begin{corollary}\label{corollaryofthetapositive}
Let $G$ be a simple Lie group with $\Theta$-positive structure, let $\Gamma_0$ be a convex cocompact Fuchsian group, and let $\rho:\Gamma_0\to G$ be a $\Theta$-positive representation. Then for any $\alpha\in \Theta$, if $\xi^{\alpha}:\Lambda(\Gamma_0)\to \mathcal{F}_{\alpha}$ is the limit map to the $\alpha$-flag manifold $\mathcal{F}_{\alpha}$ and $\delta^{\alpha}_{\rho}(\Gamma_0)$ is the $\alpha$-critical exponent, we have
\[
\dim_H \big(\xi^{\alpha}(\Lambda(\Gamma_0))\big) = \delta^{\alpha}_{\rho}(\Gamma_0).
\]
\end{corollary}

\textbf{Acknowledgements.}
The author sincerely thanks Jialun Li and Disheng Xu for introducing him to the related questions and for many fruitful discussions. He is also grateful to his advisor, Tengren Zhang, for his continuous guidance and support. Finally, the author would like to thank the organizers of the Winter Workshop on Dynamical Systems and Mathematical Physics, held at Great Bay University from December 15 to 19, 2025, which facilitated the development of this paper.

\section{Notation and Conventions}\label{Sectionofconventions}
\subsection{Linear Algebra}
Let $\langle-,-\rangle$ be a fixed Euclidean metric on $\mathbb{R}^{d}$ where the induced norm is denoted by $\|\cdot\|$, and let $\{e_1,\dots, e_d\}$ be a fixed orthonormal basis. Let $\mathsf{PGL}(d,\mathbb{R})$ be the projective general linear group, and let $\mathsf{PO}(d,\mathbb{R})$ be the projective orthogonal group with respect to $\langle-,-\rangle$. Let $\mathfrak{sl}(d,\mathbb{R})$ and $\mathfrak{so}(d,\mathbb{R})$ denote their respective Lie algebras.

By a slight abuse of notation, for any element $g \in \mathsf{PGL}(d,\mathbb{R})$, we implicitly choose a representative matrix $g \in \mathsf{GL}(d,\mathbb{R})$ such that $|\det(g)| = 1$. Note that this choice of representative does not affect the well-definedness of the subsequent definitions. For each $g\in\mathsf{PGL}(d,\mathbb{R})$ and $1\le i\le d$, let $\sigma_i(g)$ be the $i$-th singular value (of its representative). We choose a singular value decomposition $g = u(g)a(g)v(g)$ (which might not be unique), where $u(g),v(g)\in \mathsf{O}(d,\mathbb{R})$ and $a(g) = \diag(\sigma_1(g),\dots,\sigma_d(g))$.

Let $\mathfrak{a}\subset \mathfrak{sl}(d,\mathbb{R})$ be the \emph{Cartan subalgebra} consisting of diagonal trace-zero matrices. For each $0<i<d$, the $i$-th \emph{simple root} $\alpha_i\in \mathfrak{a}^*$ is defined as
\[
\alpha_i(\diag(\lambda_1,\dots, \lambda_d)) = \lambda_i-\lambda_{i+1}.
\]
Correspondingly, the $i$-th \emph{fundamental weight} is defined as 
\[
\omega_i(\diag(\lambda_1,\dots, \lambda_d)) = \lambda_1+\dots+\lambda_{i}.
\]
The \emph{positive Weyl chamber} $\mathfrak{a}^+\subset \mathfrak{a}$ is identified as
\[
\mathfrak{a}^+ := \{\diag(\lambda_1,\dots, \lambda_d)\mid \lambda_1>\dots >\lambda_d\}.
\]
The \emph{Cartan projection} $\kappa: \mathsf{PGL}(d,\mathbb{R})\to \overline{\mathfrak{a}^+}$ is defined by
\[
g\mapsto \diag(\log \sigma_1(g),\dots, \log \sigma_d(g)).
\]

We have the uniform continuity of the Cartan projection (whose proof is standard; see, for example, Canary~\cite[Lemma~29.1]{AnosovNotes}) stated as follows:

\begin{lemma}\label{lemmaofuniformcontinuity}
For any compact subset $K\subset \mathsf{PGL}(d,\mathbb{R})$, there exists a constant $C$ such that for any $g\in \mathsf{PGL}(d,\mathbb{R})$ and $h\in K$,
\[
\|\kappa(gh)-\kappa(g)\|\le C, \quad \|\kappa(hg)-\kappa(g)\| \le C.
\]
\end{lemma}

\subsection{Flag Manifold}

Let $\Delta = \{1,\dots, d-1\}$. For each non-empty subset $\theta = \{i_1, \dots, i_k\} \subset \Delta$ with $i_1 < \dots < i_k$, the \emph{$\theta$-flag manifold} $\mathcal{F}_{\theta}(\mathbb{R}^d)$ (abbreviated as $\mathcal{F}_\theta$ when there is no confusion) is defined as the space of partial flags of type $\theta$:
\[
\mathcal{F}_\theta := \{ (W_1, \dots, W_k) \mid 0 \subsetneq W_1 \subsetneq \dots \subsetneq W_k \subsetneq \mathbb{R}^d, \, \dim(W_j) = i_j \text{ for } 1 \le j \le k \}.
\]
For each $F\in \mathcal{F}_{\theta}$ and $i\in \theta$, we use $F^{(i)}$ to denote its $i$-dimensional subspace. Let $F_\theta \in \mathcal{F}_\theta$ be the standard $\theta$-flag, where $F_{\theta}^{(i_j)}$ is spanned by $\{e_1, \dots, e_{i_j}\}$. Similarly, let $F_\theta^{opp}$ be the standard opposite $\theta$-flag, where $(F_{\theta}^{opp})^{(i_j)}$ is spanned by $\{e_{d}, e_{d-1}, \dots, e_{d-i_j+1}\}$. 

Note that for each non-empty $\theta\subset \Delta$, $\mathsf{PGL}(d,\mathbb{R})$ acts transitively on $\mathcal{F}_{\theta}$. We define the standard \emph{parabolic subgroup} $\mathsf{P}_\theta$ as the stabilizer of $F_\theta$ in $\mathsf{PGL}(d,\mathbb{R})$, and the opposite parabolic subgroup $\mathsf{P}_{\theta}^{opp}$ as the stabilizer of $F_\theta^{opp}$. Let $\mathsf{U}_{\theta}$ and $\mathsf{U}_{\theta}^{opp}$ denote their respective unipotent radicals. The \emph{Levi subgroup} $\mathsf{L}_\theta$ is defined as the intersection $\mathsf{L}_\theta = \mathsf{P}_\theta \cap \mathsf{P}_{\theta}^{opp}$. Their respective Lie algebras are denoted by $\mathfrak{p}_{\theta}, \mathfrak{p}_{\theta}^{opp}, \mathfrak{u}_{\theta}, \mathfrak{u}_{\theta}^{opp}$, and $\mathfrak{l}_{\theta}$.

\vspace{3mm}

Let $\iota: \Delta \to \Delta$ be the involution defined by $i \mapsto d-i$. We denote $\theta^{opp} := \{\iota(i) \mid i \in \theta\}$, and we say a subset $\theta$ is \emph{symmetric} if $\theta = \theta^{opp}$. 
For any non-empty $\theta \subset \Delta$, two flags $F_1 \in \mathcal{F}_{\theta}$ and $F_2 \in \mathcal{F}_{\theta^{opp}}$ are said to be \emph{transverse} if for each $i \in \theta$, the subspace $F_1^{(i)}$ is transverse to $F_2^{(d-i)}$.

\vspace{3mm}

For each $0 < i < d$, an element $g \in \mathsf{PGL}(d,\mathbb{R})$ is said to be \emph{$i$-divergent} if $\sigma_i(g) > \sigma_{i+1}(g)$. We say $g$ is \emph{$\theta$-divergent} if it is $i$-divergent for all $i \in \theta$. For a $\theta$-divergent element $g$, the \emph{Cartan attractor} $U_{\theta}(g) \in \mathcal{F}_{\theta}$ is defined as $u(g)F_{\theta}$. The $\theta$-divergence condition ensures that this is well-defined. Similarly, the \emph{Cartan repeller} $V_{\theta^{opp}}(g) \in \mathcal{F}_{\theta^{opp}}$ is defined as the Cartan attractor of $g^{-1}$ (note that $g$ being $\theta$-divergent is equivalent to $g^{-1}$ being $\theta^{opp}$-divergent).

\subsection{The Projective Space}
When $\theta = \{1\}$, $\mathcal{F}_{\theta}$ is identified with the projective space $\mathbb{RP}^{d-1}$, and when $\theta = \{d-1\}$, $\mathcal{F}_{\theta}$ is identified with the space of hyperplanes, which is exactly the projective space of the dual of $\mathbb{R}^{d}$ (denoted by $(\mathbb{RP}^{d-1})^*$). We will discuss these spaces in more detail below.
\subsubsection{The metric}
We use $[v]$ to denote elements in $\mathbb{RP}^{d-1}$, where $v\in \mathbb{R}^{d}\setminus \{0\}$ is a representative. Similarly, we use $[\phi]$ to denote elements in $(\mathbb{RP}^{d-1})^* \setminus \{0\}$, where $\phi\in (\mathbb{R}^{d})^*$ is a representative. By a slight abuse of notation, under certain context we also refer to $\phi$ as a covector in $\mathbb{R}^{d}$ satisfying 
\[
\forall v\in \mathbb{R}^{d},\quad \phi(v) = \langle \phi, v\rangle,
\]
and we use $\mathcal{H}_{\phi}$ to denote the hyperplane perpendicular to $\phi$, which is identified with $[\phi]$.

We define the angular metric on $\mathbb{RP}^{d-1}$ by
\[
d_{\angle}([v_1],[v_2]) = \arccos\big(\frac{|\langle v_1,v_2\rangle|}{\|v_1\|\cdot \|v_2\|}\big).
\]
Similarly, the angular metric on $(\mathbb{RP}^{d-1})^{*}$ (also denoted by $d_{\angle}(-,-)$) is defined by
\[
d_{\angle}([\phi_1],[\phi_2]) = \arccos\big(\frac{|\langle \phi_1,\phi_2\rangle|}{\|\phi_1\|\cdot \|\phi_2\|}\big),
\]
where $\phi_1,\phi_2$ are identified as covectors in $\mathbb{R}^{d}$. Note that these definitions are independent of the choice of representatives.

\subsubsection{The volume}\label{subsectionofvolume}
The angular metric induces an unsigned volume form $d\sigma$ on $\mathbb{RP}^{d-1}$ (and similarly on $(\mathbb{RP}^{d-1})^*$), and we denote the induced measure by $m_{vol}$.

Note that for each $[v]\in \mathbb{RP}^{d-1}$, if we choose a representative $\hat v$ such that $\|\hat v\| = 1$, the tangent space at $[v]$ can be identified with the hyperplane perpendicular to $\hat v$ (denoted by $TS^{d-1}_{\hat v}$), which is tangent to the unit sphere. Under this identification, the Riemannian metric induced by the angular metric coincides with the Euclidean metric $\langle-,-\rangle$ restricted to $TS^{d-1}_{\hat v}$. For any $d-1$ tangent vectors
\[
w_1,\dots, w_{d-1}\in TS^{d-1}_{\hat v}
\]
which can be identified as vectors in $\mathbb{R}^{d}$, the unsigned volume form $d\sigma$ can be computed as follows:
\[
d\sigma|_{[v]}(w_1,\dots, w_{d-1}) = |\det (\hat v, w_1,\dots, w_{d-1})|.
\]

We record the transformation cocycle of $d\sigma$ for future reference. For any $g\in \mathsf{PGL}(d,\mathbb{R})$, recall that $g$ implicitly represents an element in $\mathsf{GL}(d,\mathbb{R})$ whose determinant has absolute value $1$. Let $\mathcal{D}g$ denote the differential map of $g$, let $\pi_{\perp}$ be the orthogonal projection from $\mathbb{R}^{d}$ onto $TS^{d-1}_{\frac{g\hat v}{\|g \hat v\|}}$, and let $\|\cdot \|$ denote the natural norm on the exterior spaces. Then,
\begin{equation}\label{equationofcocycleforform}
\begin{split}
d\sigma|_{g[v]}(\mathcal{D}g w_1,\dots, \mathcal{D}g w_{d-1}) &= \left|\det \left(\frac{g\hat v}{\|g\hat v\|}, \pi_{\perp}\left(\frac{gw_1}{\|g\hat v\|}\right),\dots, \pi_{\perp}\left(\frac{gw_{d-1}}{\|g\hat v\|}\right)\right)\right| \\
 &= \left\| \frac{g\hat v}{\|g\hat v\|}\wedge \frac{gw_1}{\|g\hat v\|}\wedge\dots \wedge \frac{gw_{d-1}}{\|g\hat v\|}\right\| \\
 &= \frac{\| g\hat v\wedge gw_1\wedge \cdots \wedge  g w_{d-1}\|}{\|g\hat v\|^{d}}\\
 &=\frac{\|\hat{v} \wedge w_1\wedge \dots\wedge w_{d-1}\|}{\|g\hat v\|^{d}} \\
  &=\frac{d\sigma|_{[v]}(w_1,\dots , w_{d-1})}{\|g\hat v\|^{d}} \\
 & = \big(\frac{\|v\|}{\|g v\|} \big)^{d} d\sigma|_{[v]}(w_1,\dots , w_{d-1}).
\end{split}
\end{equation}
Note that the final expression is independent of the choice of the representative $v$.

As a result, we obtain the equation for the pull-back Radon-Nikodym cocycle of $m_{vol}$:
\begin{equation}\label{equationofcocycleformeasure}
\frac{d g^* m_{vol}}{dm_{vol}}([v]) = \left(\frac{\|v\|}{\|gv\|}\right)^{d}.
\end{equation}
Or equivalently, for the push-forward, we have
\begin{equation}\label{equationofcocycleformeasure2}
\frac{d g_* m_{vol}}{dm_{vol}}([v]) = \left(\frac{\|v\|}{\|g^{-1}v\|}\right)^{d}.
\end{equation}

\subsubsection{Ellipses}\label{sectionofellipse}
Let $v_0, w_1, \dots, w_{d-1}$ be orthonormal vectors. Let $TS^{d-1}_{v_0}$ be the hyperplane perpendicular to $v_0$, tangent to the unit sphere at $v_0$. For any $d-1$ non-negative numbers $c_1, \dots, c_{d-1}$, let $\tilde{\mathcal{E}}(v_0; w_1, \dots, w_{d-1}; c_1, \dots, c_{d-1}) \subset TS^{d-1}_{v_0}$ be the ellipsoid centered at $v_0$ with axes in the directions of $w_1, \dots, w_{d-1}$ having corresponding lengths $2c_1, \dots, 2c_{d-1}$. Let $\mathcal{E}(v_0; w_1, \dots, w_{d-1}; c_1, \dots, c_{d-1}) \subset \mathbb{RP}^{d-1}$ be its projectivization.

For any $c \ge 0$ and unit vector $v_0$, choose $w_1', \dots, w_{d-1}'$ to be an orthonormal basis for the orthogonal complement of $v_0$. We define
\[
\mathcal{B}(v_0,c) = \mathcal{E}(v_0; w'_1, \dots, w'_{d-1}; c, \dots, c),
\]
which is independent of the choice of $w_1', \dots, w_{d-1}'$.

The above definitions of $\mathcal{E}(v_0; w_1, \dots, w_{d-1}; c_1, \dots, c_{d-1})$ and $\mathcal{B}(v_0,c)$ carry over to the case when $v_0,w_1,\dots,w_{d-1}$ are orthogonal projective points in $\mathbb{RP}^{d-1}$ by picking their representatives.

We present a useful lemma, which is straightforward to verify:
\begin{lemma}\label{lemmaofmapofellipse}
For any $c\ge 0$ and $\{1, d-1\}$-divergent element $g \in \mathsf{PGL}(d,\mathbb{R})$, we have
\[
g\mathcal{B}\big( (V_{d-1}(g))^{\perp}, c \big) = \mathcal{E}\left(U_1(g); u(g)[e_2], \dots, u(g)[e_d]; c\frac{\sigma_2(g)}{\sigma_1(g)}, \dots, c\frac{\sigma_d(g)}{\sigma_1(g)}\right),
\]
where $(V_{d-1}(g))^{\perp} = v(g)^{-1} [e_1]$ is the point in $\mathbb{RP}^{d-1}$ orthogonal to the hyperplane $V_{d-1}(g)$.
\end{lemma}

\section{Projective Anosov Subgroups}
\subsection{Definition and Representation of Projective Anosov Subgroups}

\begin{definition}\label{definitionofdivergentgroup}
Let $\theta \subset \Delta$ be a non-empty subset. A sequence $(\gamma_n)_{n\in \mathbb{Z}^{>0}} \subset \mathsf{PGL}(d,\mathbb{R})$ is called \emph{$\theta$-divergent} if for any $i\in \theta$, 
\[ 
\frac{\sigma_i(\gamma_n)}{\sigma_{i+1}(\gamma_n)} \to \infty \quad \text{as } n\to \infty. 
\] 
A subgroup $\Gamma \subset \mathsf{PGL}(d,\mathbb{R})$ is called \emph{$\theta$-divergent} if every escaping sequence in $\Gamma$ is $\theta$-divergent.
\end{definition}

Suppose $\Gamma \subset \mathsf{PGL}(d,\mathbb{R})$ is a $\theta$-divergent subgroup. The \emph{limit set} of $\Gamma$ is defined as
\[
\Lambda^{\theta}(\Gamma) = \{F \in \mathcal{F}_{\theta} \mid F \text{ is the accumulation point of } U_{\theta}(\gamma_n) \text{ for some divergent sequence } (\gamma_n) \text{ in } \Gamma\},
\]
which is a $\Gamma$-invariant and compact set.

Note that if $\Gamma$ is $\theta$-divergent, then it is also $\theta^{opp}$-divergent. Therefore, without loss of generality, we will always assume that $\theta$ is symmetric.

The notion of a transverse subgroup was introduced by Kapovich--Leeb--Porti~\cite{kapovich2017anosov} as follows:

\begin{definition}\label{definitionoftransverse}
A subgroup $\Gamma \subset \mathsf{PGL}(d,\mathbb{R})$ is called \emph{$\theta$-transverse} if it is $\theta$-divergent and any two distinct flags in $\Lambda^{\theta}(\Gamma)$ are transverse.
\end{definition}

Notice that if a $\theta$-transverse subgroup $\Gamma \subset \mathsf{PGL}(d,\mathbb{R})$ is non-elementary, then $\Lambda^{\theta}(\Gamma)$ is a perfect set.

There are numerous definitions of Anosov subgroups, for example the singular value gap definition as we mentioned in the Introduction. However, for convenience, in the main body of the paper we adopt the one by Kapovich--Leeb--Porti~\cite[Definition~1.5]{kapovich2014morse}.

\begin{definition}\label{definitionofanosov}
Let $\theta_1 = \{1, d-1\}$. A hyperbolic subgroup $\Gamma$ of $\mathsf{PGL}(d,\mathbb{R})$ is called \emph{projective Anosov} if $\Gamma$ is $\theta_1$-transverse and $\Lambda^{\theta_1}(\Gamma)$ is $\Gamma$-equivariantly homeomorphic to the Gromov boundary $\partial_\infty \Gamma$.
\end{definition}

Here is a useful criterion for Anosov subgroups:

\begin{theorem}[{Guichard--Wienhard~\cite[Theorem~5.3]{GW1}, Kapovich--Leeb--Porti~\cite[Theorem~1.5]{anosovmorseandbuilding}, and Bochi--Potrie--Sambarino~\cite[Theorem~8.4]{dominatedsplitting}}]\label{theoremofcriterionofsingularvaluegaps}

Let $\Gamma\subset \mathsf{PGL}(d,\mathbb{R})$ be a finitely generated subgroup, and let $d_{word}$ be the word metric on $\Gamma$ with respect to a chosen finite set of generators. Then $\Gamma$ is projective Anosov if and only if there exist constants $a>1$ and $A>0$ such that for any $\gamma\in \Gamma$,
\[
\frac{1}{a}d_{word}(\mathrm{id},\gamma)-A \le \log \frac{\sigma_1(\gamma)}{\sigma_{2}(\gamma)} \le a d_{word}(\mathrm{id},\gamma)+A.
\]
\end{theorem}

A typical example of a projective Anosov subgroup is the holonomy representation of a closed hyperbolic manifold, which is a cocompact lattice preserving a convex round cone. It turns out that a large class of Anosov subgroup can be realized as a discrete subgroup acting cocompactly on some convex set (though the regularity of the set's boundary is generally much weaker than that of a round cone).

Let $V$ be a vector space. An open convex domain (for brevity, we will simply refer to it as a \emph{convex domain} for the remainder of this paper) $\Omega\subset \mathbb{P}(V)$ is called \emph{proper} if it is bounded in some affine chart. A point $x\in \partial \Omega$ is called \emph{$C^1$} if there exists a unique supporting hyperplane to $\Omega$ at $x$ (and we use $H_x$ to denote it), and is called \emph{extreme} if there does not exist any non-trivial open segment in $\partial \Omega$ containing $x$. 

\begin{definition}\label{definitionofregularconvexcocompact}
Let $V$ be a real vector space and $\Omega$ be a proper convex domain in $\mathbb{P}(V)$. A discrete subgroup $\Gamma_0\subset \operatorname{Aut}(\Omega)$ is called \emph{regular convex cocompact} if there exists a non-empty closed convex subset $\mathcal{C}\subset \Omega$ such that:
\begin{enumerate}
    \item $\gamma \mathcal{C} = \mathcal{C}$ for all $\gamma\in \Gamma_0$;
    \item $\mathcal{C}/\Gamma_0$ is compact;
    \item Every point in $\overline{\mathcal{C}}\cap \partial \Omega$ is a $C^1$ extreme point.  
\end{enumerate}
\end{definition}

With the preceding notation, we have the following result:
\begin{theorem}[{Zimmer~\cite[Theorem~1.22]{zimmerprojectiveconvex} and Canary--Zhang--Zimmer~\cite[Proposition~3.5]{CaZhZi}}]\label{theoremofregularcocompact}
If $\Gamma_0\subset \mathsf{PGL}(V)$ is an irreducible regular convex cocompact subgroup, then $\Gamma_0$ is projective Anosov. Moreover, the limit set of $\Gamma_0$ is precisely $\Lambda^{1,\dim V-1}(\Gamma_0) = \{(x,H_x)\mid x\in \overline{\mathcal{C}}\cap \partial \Omega\}$.
\end{theorem}

Furthermore, every irreducible torsion-free projective Anosov subgroup can be realized as an irreducible regular convex cocompact subgroup. We present a short proof of this fact, as the statements we cite differ slightly from our formulation.
\begin{theorem}[{Zimmer~\cite[Theorem~1.10]{zimmerprojectiveconvex} and Canary--Zhang--Zimmer~\cite[Theorem~4.2]{CaZhZi}}]\label{theoremofrealization}
Let $\Gamma\subset \mathsf{PGL}(d,\mathbb{R})$ be an irreducible torsion-free projective Anosov subgroup. There exist a vector space $V$, a proper convex domain $\Omega\subset \mathbb{P}(V)$, and an irreducible regular convex cocompact subgroup $\Gamma_0\subset \operatorname{Aut}(\Omega)$ such that:
\begin{enumerate}
    \item There exists a group isomorphism $\rho:\Gamma_0\to \Gamma$.
    \item There exists an equivariant homeomorphism from $\Lambda^{1,\dim V-1}(\Gamma_0)$ to $\Lambda^{1,d-1}(\Gamma)$, denoted by $\xi$, such that it induces bi-Lipschitz equivariant homeomorphism from $\Lambda^{1}(\Gamma_0)$ to $\Lambda^1(\Gamma)$ (respectivly, from $\Lambda^{\dim V - 1}(\Gamma_0)$ to $\Lambda^{d-1}(\Gamma)$), denoted by $\xi^1$ (respectively, by $\xi^{d-1}$).
\end{enumerate}
\end{theorem}
\begin{proof}
By \cite[Theorem~4.2]{CaZhZi}, there exist a vector space $\tilde{V}$, a proper convex domain $\tilde{\Omega}\subset \mathbb{P}(\tilde{V})$, a discrete subgroup $\tilde{\Gamma}_0\subset \operatorname{Aut}(\tilde{\Omega})$, and an isomorphism $\tilde{\rho}:\tilde{\Gamma}_0\to \Gamma$ such that for any $\gamma\in \tilde{\Gamma}_0$, $\frac{\sigma_1(\tilde{\rho}(\gamma))}{\sigma_2(\tilde{\rho}(\gamma))} = \frac{\sigma_1(\gamma)}{\sigma_2(\gamma)}$ (here we abuse the notation $\sigma_i$ to denote singular values on different vector spaces). By Theorem~\ref{theoremofcriterionofsingularvaluegaps}, $\tilde{\Gamma}_0$ is projective Anosov in $\mathsf{PGL}(\tilde{V})$. Moreover, by \cite[Theorem~4.2]{CaZhZi}, there is an equivariant homeomorphism
$$
\tilde{\xi} : \Lambda^{1,\dim \tilde{V}-1}(\tilde{\Gamma}_0) \to \Lambda^{1,d-1}(\Gamma),
$$
which induces bi-Lipschitz maps $\tilde{\xi}^1 : \Lambda^{1}(\tilde{\Gamma}_0) \to \Lambda^{1}(\Gamma)$ and $\tilde{\xi}^{d-1} : \Lambda^{\dim \tilde{V} - 1}(\tilde{\Gamma}_0) \to \Lambda^{d-1}(\Gamma)$. This bi-Lipschitz property holds because its construction in Proposition~4.3 therein arises from an algebraic embedding.

Let $V\subset \tilde{V}$ be the subspace spanned by $\Lambda^{1}(\tilde{\Gamma}_0)$. Let $\Omega' = \tilde{\Omega}\cap \mathbb{P}(V)$, and let $\Gamma_0$ be the restriction of $\tilde{\Gamma}_0$ to $\mathsf{PGL}(V)$ which acts irreducibly. As $\Gamma$ is assumed to be torsion-free, $\tilde{\Gamma}_0$ is also torsion-free, so the restricted action is faithful, and $\tilde{\rho}$ induces an isomorphism $\rho$ from $\Gamma_0$ to $\Gamma$. 

Since $\Gamma_0$ preserves $\Lambda^1(\tilde{\Gamma}_0)$, by \cite[Lemma~9.2]{canary2020topological}, $\Gamma_0$ is $P_1$-divergent in $\mathsf{PGL}(V)$. Note that from Theorem~\ref{theoremofregularcocompact} we can define the map
$$
\tilde{\pi}:\Lambda^{1,\dim \tilde{V}-1}(\tilde{\Gamma}_0)\to \mathcal{F}_{1,\dim V - 1}(V), \quad (x,H_x)\mapsto (x, H_x\cap V),
$$
which is injective and preserves the required bi-Lipschitz properties. Let $\Lambda^{1,\dim V-1}(\Gamma_0) = \tilde{\pi}(\Lambda^{1,\dim \tilde{V}-1}(\tilde{\Gamma}_0))$. It is straightforward to verify that this is the transverse limit set of $\Gamma_0$, and there is a continuous equivariant map $\xi$ from it to $\Lambda^{1,d-1}(\Gamma)$, which is equivariantly homeomorphic to $\partial_{\infty}(\Gamma)$.  Thus, by Definition~\ref{definitionofanosov}, $\Gamma_0$ is projective Anosov.

To conclude, note that $\Gamma_0$ is an irreducible projective Anosov subgroup preserving the proper convex domain $\Omega'$. By \cite[Theorem~1.27]{zimmerprojectiveconvex}, there further exists a proper convex domain $\Omega$ such that $(V,\Omega,\Gamma_0,\rho)$ satisfies our requirements.
\end{proof}

\begin{definition}\label{definitionofanosovrepresentation}
Let $\Gamma_0$ be a regular convex cocompact subgroup with respect to $(V,\Omega,\mathcal{C})$. A faithful representation $\rho:\Gamma_0\to \mathsf{PGL}(d,\mathbb{R})$ is called a \emph{projective Anosov representation} if $\rho(\Gamma_0)$ is projective Anosov and there exists an equivariant homeomorphism $\xi:\Lambda^{1,\dim V-1}(\Gamma_0)\to \Lambda^{1,d-1}(\rho(\Gamma_0))$. The map $\xi$ is called the \emph{limit map} of $\rho$, whose induced maps on $\Lambda^1(\Gamma_0)$ (respectively, $\Lambda^{\dim V-1}(\Gamma_0)$) are denoted by $\xi^1$ (respectively, $\xi^{d-1}$).
\end{definition}

To conclude this section, we present some dynamical properties for later use.

\begin{lemma}[{Kapovich--Leeb--Porti~\cite{kapovich2017anosov} and Canary--Zhang--Zimmer~\cite[Proposition 2.3, Proposition 2.6 and Observation 6.1]{CaZhZi3}}]\label{lemmaofcartanproperty}
    Let $\Gamma_0$ be a regular convex cocompact group with respect to $(V,\Omega,\mathcal{C})$, and let $\rho:\Gamma_0\to \mathsf{PGL}(d,\mathbb{R})$ be a projective Anosov representation with limit map $\xi$.

   Fix $b_0 \in \mathcal{C} $. If $(\gamma_n)_{n \in \mathbb{N}}$ is a sequence in $\Gamma_0$ such that $\gamma_n b_0 \to x \in \Lambda^1(\Gamma_0)$ and $\gamma_n^{-1} b_0 \to y \in \Lambda^1(\Gamma_0)$ as $n \to \infty$, then we have:
\begin{enumerate}
    \item \textbf{(Cartan property)} $\big(U_1(\rho(\gamma_n)), U_{d-1}(\rho(\gamma_n)) \big) \to \xi(x,H_x)$.
    \item \textbf{(Strongly dynamics preserving property)} For any compact subset $\mathcal{K} \subset \mathcal{F}_{1,d-1}$ transverse to $\xi(y,H_y)$, the sequence $\rho(\gamma_n)\mathcal{K}$ converges to $\xi(x,H_x)$ uniformly.
\end{enumerate}
    
\end{lemma}

\subsection{Geometry of regular convex cocompact groups}

First, we discuss some projective geometry. Let $V$ be a vector space, and let $\Omega \subset \mathbb{P}(V)$ be a proper convex domain.

Recall that $\Omega$ admits an $\operatorname{Aut}(\Omega)$-invariant Finsler metric, the \emph{Hilbert metric}, defined by
\[
d_{\Omega}(x,y) = \frac{1}{2} \log [a,x,y,b],
\]
where $a, b \in \partial \Omega$ are the intersection points of the line through $x$ and $y$ with $\partial \Omega$ such that $a, x, y, b$ appear in that order, and $[a,x,y,b]$ is the cross-ratio, which can be computed as $\frac{|y-a||b-x|}{|x-a||b-y|}$ in any affine chart. When $\Omega$ is an ellipsoid (the projectivization of a round convex cone), $(\Omega, d_{\Omega})$ can be identified isometrically with the hyperbolic space of constant curvature $-1$. Also recall that for any $x, y \in \Omega$, the straight line segment $[x,y]$ connecting them is a geodesic with respect to $d_{\Omega}$.

Now we discuss more on regular convex cocompact subgroups.  Let $(V,\Omega,\Gamma_0,\mathcal{C})$ be the objects described in Definition~\ref{definitionofregularconvexcocompact}. By Theorem~\ref{theoremofregularcocompact}, $\Gamma_0$ is projective Anosov; thus, it is hyperbolic and finitely generated (we always implicitly fix a finite generator set). By Selberg's lemma, we can assume without loss of generality (by passing to a finite-index subgroup) that $\Gamma_0$ is torsion-free. In this case, it can be verified that $\Gamma_0$ acts freely and by isometries on $(\Omega,d_{\Omega})$, and thus the quotient $\Omega/\Gamma_0$ is a Finsler manifold.

By the Švarc--Milnor lemma, $(\mathcal{C},d_{\Omega})$ is a $\delta$-hyperbolic space for some $\delta>0$. And for any $b_0\in \mathcal{C}$, the map
$$
\gamma\mapsto \gamma b_0
$$
is a quasi-isometry from $(\Gamma_0, d_{word})$ to $(\mathcal{C},d_{\Omega})$, meaning that for a fixed $b_0$, there exist constants $a>1$ and $A>0$ such that for any $\gamma\in \Gamma_0$,
$$
\frac{1}{a} d_{word}(\mathrm{id},\gamma)-A\le d_{\Omega}(b_0,\gamma b_0) \le a d_{word}(\mathrm{id},\gamma)+A.
$$

From the Morse lemma for quasi-geodesics in $\delta$-hyperbolic spaces (see, for example, Canary~\cite[Theorem~2.2]{AnosovNotes}), we have the following lemma.

\begin{lemma}\label{lemmaofquasigeodesics}
Fix $b_0 \in \mathcal{C}$. There exist constants $R > 0$, $a>1$, and $A>0$ such that for any $x \in \overline{\mathcal{C}} \cap \partial\Omega$, there exists a sequence $(\gamma_n)_{n \in \mathbb{Z}^{>0}}$ in $\Gamma_0$ satisfying:
\begin{itemize}
    \item $\gamma_n b_0 \to x$ as $n \to +\infty$.
    \item The sequence $(\gamma_n)_{n \in \mathbb{Z}^{>0}}$ represents an $(a,A)$-quasi-geodesic ray in the Cayley graph of $\Gamma_0$, i.e., for any $n,m>0$,
    \[
    \frac{1}{a}|n-m|-A \le d_{word}(\gamma_n, \gamma_m) \le a |n-m| +A.
    \]

    \item $[b_0,x) \subset \bigcup_{n \in \mathbb{Z}^{>0}} B_{d_{\Omega}}(\gamma_n b_0, R)$.
\end{itemize}
\end{lemma}

\subsection{Geometry of the shadows}
Now, we introduce the concept of shadows.

\begin{definition}
    For any $b \in \overline{\Omega}$, $z \in \Omega$, and $R > 0$, the shadow of $z$ viewed from $b$ with radius $R$ is defined as
    $$
    \mathcal{O}_{R}(b,z) := \{x \in \partial \Omega \mid [b,x) \cap \overline{B}_{d_{\Omega}}(z,R) \neq \emptyset\}.
    $$

    If $\Gamma_0$ is a regular convex cocompact subgroup with respect to $(V, \Omega, \mathcal{C})$, we let
$$
\hat{\mathcal{O}}_R(b,z) := \mathcal{O}_R(b,z) \cap \Lambda^1(\Gamma_0).
$$
    and
\[
\tilde{\mathcal{O}}_R(b,z) := \{(x,H_x)\mid x\in \hat{\mathcal{O}}_R(b,z) \}.
\]

\end{definition}

We recall a Vitali-type covering lemma for shadows, due to Roblin.

\begin{lemma}[{Roblin~\cite[p.~23]{shadowvitali}}]\label{lemmaofshadowvitali}
    For any $b_0 \in \Omega$ and $R > 0$, if $I \subset \Gamma_0$, then there exists a subset $J \subset I$ such that the family $\{\mathcal{O}_{R}(b_0, \gamma b_0) \mid \gamma \in J\}$ is pairwise disjoint, and
    $$
    \bigcup_{\gamma \in I} \mathcal{O}_{R}(b_0, \gamma b_0) \subset \bigcup_{\gamma \in J} \mathcal{O}_{5R}(b_0, \gamma b_0).
    $$
\end{lemma}

For any $1$-divergent $g \in \mathsf{PGL}(d,\mathbb{R})$ and $C > 0$, we define (see Section~\ref{sectionofellipse})
\[
\mathcal{E}(g,C) := \mathcal{E}\left(u(g)[e_1]; u(g)[e_2], \dots, u(g)[e_d]; C \frac{\sigma_2(g)}{\sigma_1(g)}, \dots, C\frac{\sigma_d(g)}{\sigma_1(g)}\right),
\]
which is independent of the choice of the singular value decomposition of $g$. Note that if $\Gamma$ is a projective Anosov subgroup, then all but finitely many elements of $\Gamma$ are $\{1,d-1\}$-divergent.

\begin{revblock}
With respect to the ellipsoids, we establish a stronger version of Lemma~\ref{lemmaofshadowvitali}.

\begin{proposition}\label{propositionofellipsevitali}
    Let $\Gamma_0$ be a torsion-free regular convex cocompact group with respect to $(V,\Omega,\mathcal{C})$, and let $\rho:\Gamma_0\to \mathsf{PGL}(d,\mathbb{R})$ be a projective Anosov representation with limit map $\xi$. Fix $b_0 \in \mathcal{C}$ and $C > 0$. Then there exist a finite subset $F_0 \subset \Gamma_0$ and $K = K(C) \ge 2$ such that for any $I \subset \Gamma_0 \setminus F_0$, there exists $J \subset I$ for which $\{\mathcal{E}(\rho(\gamma), C) \mid \gamma \in J\}$ is pairwise disjoint in $\mathbb{RP}^{d-1}$ and
    \[
    \bigcup_{\gamma \in I} \mathcal{E}(\rho(\gamma), C) \subset \bigcup_{\gamma \in J} \mathcal{E}(\rho(\gamma), KC).
    \]
\end{proposition}

\begin{proof}
We first claim that there exist a finite subset $F_0 \subset \Gamma_0$ (containing all $\gamma$ with $\rho(\gamma)$ not $\{1,d-1\}$-divergent) and $K \ge 2$ such that for all $\gamma_1,\gamma_2 \in \Gamma_0 \setminus F_0$,
\begin{equation}\label{eq:ellipsoid_doubling}
    \begin{aligned}
    d_{\Omega}(b_0,\gamma_1 b_0) \le d_{\Omega}(b_0,\gamma_2 b_0)
    \ &\text{ and }\ 
    \mathcal{E}(\rho(\gamma_1),C) \cap \mathcal{E}(\rho(\gamma_2),C) \neq \emptyset \\
    &\implies \mathcal{E}(\rho(\gamma_2),C) \subset \mathcal{E}(\rho(\gamma_1),KC).
    \end{aligned}
\end{equation}

Suppose \eqref{eq:ellipsoid_doubling} fails, then from Lemma~\ref{lemmaofmapofellipse}, there exist sequences $(\gamma_{1,n}),(\gamma_{2,n})$ in $\Gamma_0$ with $n < d_{\Omega}(b_0,\gamma_{1,n} b_0) \le d_{\Omega}(b_0,\gamma_{2,n} b_0)$ such that,
\begin{equation}\label{eq:vitali_pullback}
    \begin{aligned}
    \mathcal{B}\big(V_{d-1}(\rho(\gamma_{1,n}))^{\perp},C\big) \cap \rho(h_n)\mathcal{B}\big(V_{d-1}(\rho(\gamma_{2,n}))^{\perp},C\big) \neq \emptyset, \\
    \rho(h_n)\mathcal{B}\big(V_{d-1}(\rho(\gamma_{2,n}))^{\perp},C\big) \not\subset \mathcal{B}\big(V_{d-1}(\rho(\gamma_{1,n}))^{\perp},nC\big).
    \end{aligned}
\end{equation}
Here $h_n = \gamma_{1,n}^{-1} \gamma_{2,n}$.

Passing to a subsequence, assume $\gamma_{i,n} b_0 \to x_i \in \Lambda^1(\Gamma_0)$ and $\gamma_{i,n}^{-1} b_0 \to y_i \in \Lambda^1(\Gamma_0)$ for $i=1,2$. By Lemma~\ref{lemmaofcartanproperty}~(1), $V_{d-1}(\rho(\gamma_{i,n})) \to \xi^{d-1}(H_{y_i})$, so $\mathcal{B}\big(V_{d-1}(\rho(\gamma_{1,n}))^{\perp},nC\big)$ eventually contains every compact subset of $\mathbb{RP}^{d-1}$ transverse to $\xi^{d-1}(H_{y_1})$. To contradict the second relation in \eqref{eq:vitali_pullback}, it suffices to show that $\rho(h_n)\mathcal{B}\big(V_{d-1}(\rho(\gamma_{2,n}))^{\perp},C\big)$ eventually lies in such a compact subset.

If $(h_n)$ has a bounded subsequence, discreteness allows us to pass to a subsequence with $h_n = h$ constant. Then $y_2 = h^{-1}y_1$ and $V_{d-1}(\rho(\gamma_{2,n})) \to \rho(h)^{-1}\xi^{d-1}(H_{y_1})$. Hence, $\rho(h)\mathcal{B}\big(V_{d-1}(\rho(\gamma_{2,n}))^{\perp},C\big)$ converges to a compact subset transverse to $\xi^{d-1}(H_{y_1})$.

If $(h_n)$ diverges, then we pass to a subsequence with $h_n b_0 \to z_+ \in \Lambda^1(\Gamma_0)$ and $h_n^{-1} b_0 \to z_- \in \Lambda^1(\Gamma_0)$. We must have $y_2 = z_-$, otherwise $(y_2,z_-)\subset\Omega$ (see Theorem~\ref{theoremofregularcocompact}), so the $d_{\Omega}$-geodesics $[\gamma_{2,n}^{-1}b_0, \gamma_{2,n}^{-1}\gamma_{1,n}b_0]$ stay within a uniform distance $D$ of $b_0$, and the triangle inequality gives
\[
    d_{\Omega}(b_0,\gamma_{1,n}b_0) = d_{\Omega}(\gamma_{2,n}^{-1}b_0,\gamma_{2,n}^{-1}\gamma_{1,n}b_0) \ge d_{\Omega}(b_0,\gamma_{2,n}b_0) + d_{\Omega}(b_0,h_n b_0) - 2D,
\]
contradicting $d_{\Omega}(b_0,\gamma_{1,n}b_0)\le d_{\Omega}(b_0,\gamma_{2,n}b_0)$ as $d_{\Omega}(b_0,h_n b_0)\to\infty$. Lemma~\ref{lemmaofcartanproperty}~(2) and $y_2= z_-$ show $\rho(h_n)\mathcal{B}\big(V_{d-1}(\rho(\gamma_{2,n}))^{\perp},C\big) \to \xi^1(z_+)$ uniformly. By the first relation in \eqref{eq:vitali_pullback}, $\xi^1(z_+) \in \mathcal{B}\big(\xi^{d-1}(H_{y_1})^{\perp},C\big)$, which is transverse to $\xi^{d-1}(H_{y_1})$.

Thus in both cases we have proved \eqref{eq:ellipsoid_doubling}.

Finally, given \eqref{eq:ellipsoid_doubling}, the proposition follows from the standard argument of the Vitali covering lemma. For any $I \subset \Gamma_0 \setminus F_0$, proper discontinuity allows us to enumerate $I = \{\gamma_1,\gamma_2,\dots\}$ with $d_{\Omega}(b_0,\gamma_k b_0)$ non-decreasing. Inductively include $\gamma_k$ in $J$ whenever $\mathcal{E}(\rho(\gamma_k),C)$ is disjoint from $\mathcal{E}(\rho(\gamma),C)$ for all previously chosen $\gamma\in J$. Then $\{\mathcal{E}(\rho(\gamma),C) \mid \gamma\in J\}$ is pairwise disjoint. By construction, for every $\gamma\in I\setminus J$, the ellipsoid $\mathcal{E}(\rho(\gamma),C)$ intersects $\mathcal{E}(\rho(\gamma'),C)$ for some earlier $\gamma'\in J$, whence $\mathcal{E}(\rho(\gamma),C) \subset \mathcal{E}(\rho(\gamma'),KC)$ by \eqref{eq:ellipsoid_doubling}.
\end{proof}
\end{revblock}

Now, we present a correspondence between shadows and ellipses.

\begin{lemma}\label{lemmaofellipseandshadows}
    Let $\Gamma_0$ be a torsion-free regular convex cocompact group with respect to $(V,\Omega,\mathcal{C})$, and let $\rho:\Gamma_0\to \mathsf{PGL}(d,\mathbb{R})$ be a projective Anosov representation with limit map $\xi$. 

    Fix $b_0\in \mathcal{C}$. 
    \begin{enumerate}

        \item For any $C>0$, there exists $R>0$ such that for all $\gamma\in \Gamma_0 \setminus \{\mathrm{id}\}$,
        \[
            \mathcal{E}\left(\rho(\gamma),C\right)\cap \Lambda^1(\rho(\Gamma_0))\subset \xi^1( \hat{\mathcal{O}}_R(b_0,\gamma b_0)).
        \]
        \item For any $R>0$, there exists $C>0$ such that for all but finitely many $\gamma\in \Gamma_0 \setminus \{\mathrm{id}\}$,
        $$
        \xi^1( \hat{\mathcal{O}}_R(b_0,\gamma b_0))\subset \mathcal{E}(\rho(\gamma), C).
        $$
    \end{enumerate}
\end{lemma}

\begin{proof}
    \textbf{Proof of (1):} By equivariance, it is equivalent to prove that there exist constants $R>0$  such that for any $\gamma\in \Gamma_0$,
    $$
    \rho(\gamma)^{-1}\mathcal{E}\left(\rho(\gamma),C\right)\cap \Lambda^1(\rho(\Gamma_0))\subset \xi^1( \hat{\mathcal{O}}_R(\gamma^{-1}b_0, b_0)).
    $$
    By Lemma~\ref{lemmaofmapofellipse}, this is equivalent to prove that there exists $R>0$ such that for any $\gamma\in \Gamma_0$,
    \begin{equation}\label{equationofshadowandcartanballs}
        \mathcal{B}\left(V_{d-1}(\rho(\gamma))^{\perp},C\right)\cap \Lambda^1(\rho(\Gamma_0))\subset \xi^1( \hat{\mathcal{O}}_R(\gamma^{-1}b_0, b_0)).
    \end{equation}

    Suppose this is not true, then we can find a sequence $(\gamma_n)_{n \in \mathbb{N}}$ with respect to $R_n = n$ such that
    \begin{equation}\label{equationofleftinclusionfailure}
        \mathcal{B}\left(V_{d-1}(\rho(\gamma_n))^{\perp},C\right)\cap \Lambda^1(\rho(\Gamma_0)) \not\subset \xi^1( \hat{\mathcal{O}}_{n}(\gamma_n^{-1}b_0, b_0)).
    \end{equation}

    The sequence $(\gamma_n)_{n \in \mathbb{N}}$ must be unbounded; otherwise, $d_{\Omega}(\gamma_n^{-1} b_0, b_0)$ would be bounded, and there would exist $N>0$ such that for all $n>N$, the shadow $\hat{\mathcal{O}}_{n}(\gamma_n^{-1}b_0, b_0) = \Lambda^1(\Gamma_0)$, which would trivially contain the left-hand side. Thus, by passing to a subsequence, we may assume $\gamma_n^{-1} b_0\to y\in \Lambda^1(\Gamma_0)$. 
    Note that as $n \to \infty$, $B_{d_{\Omega}}(b_0,n)$ exhausts $\Omega$, and for any $x \neq y\in \Lambda^1(\Gamma_0)$, $(x,y) \subset \Omega$,  so $\hat{\mathcal{O}}_n(\gamma_n^{-1} b_0, b_0)\setminus \{y\}$ exhausts the set $\Lambda^1(\Gamma_0) \setminus \{y\}$. Consequently, its image under $\xi^1$ exhausts $\Lambda^1(\rho(\Gamma_0)) \setminus \{\xi^1(y)\}$, which consists exactly of the points in $\Lambda^1(\rho(\Gamma_0))$ that are transverse to $\xi^{d-1}(H_y)$. 
    On the other hand, by the Cartan property in Lemma~\ref{lemmaofcartanproperty}, $V_{d-1}(\rho(\gamma_n)) \to \xi^{d-1}(H_y)$. Therefore, the regions $\mathcal{B}\left(V_{d-1}(\rho(\gamma_n))^{\perp},C\right)$ will eventually be swallowed by the exhausting domains transverse to $\xi^{d-1}(H_y)$. Thus, for $n \gg 1$, Equation~(\ref{equationofleftinclusionfailure}) cannot hold, which is a contradiction.

    \textbf{Proof of (2):} 

    Fix $R>0$ and suppose that such a $C$ does not exist. Then, similar to the proof of (1), we can find a diverging sequence $C_n \to \infty$ and a diverging sequence of distinct elements $(\gamma_n)_{n \in \mathbb{N}}$ in $\Gamma_0$ (the divergence is required because the statement holds for all but finitely many $\gamma$) such that
    \begin{equation}\label{failureofrightinclusion}
        \xi^1( \hat{\mathcal{O}}_R(\gamma_n^{-1}b_0, b_0))\not \subset \mathcal{B}\left(V_{d-1}(\rho(\gamma_n))^{\perp},C_n\right).
    \end{equation}
    By passing to a subsequence, we assume $\gamma_n^{-1} b_0 \to y \in \Lambda^1(\Gamma_0)$. 
    The shadow $\hat{\mathcal{O}}_R(\gamma_n^{-1}b_0, b_0)$ converges to $\hat{\mathcal{O}}_R(y, b_0)$, which is a compact subset strictly bounded away from $y$. Its image is therefore a compact subset transverse to $\xi^{d-1}(H_y)$. 
    Simultaneously, as $V_{d-1}(\rho(\gamma_n)) \to \xi^{d-1}(H_y)$, the regions $\mathcal{B}\left(V_{d-1}(\rho(\gamma_n))^{\perp},C_n\right)$ exhaust the entire space of directions transverse to $\xi^{d-1}(H_y)$. Thus, they must eventually contain the image of the shadow, contradicting Equation~(\ref{failureofrightinclusion}).
\end{proof}

\section{Patterson-Sullivan Theory}
 We will follow Canary--Zhang--Zimmer~\cite{CaZhZi3} to introduce the constructions of Patterson--Sullivan measures.

\subsection{Definition and Properties of the Patterson-Sullivan Measure}
We refer to the notation and conventions established in Section~\ref{Sectionofconventions}. For any non-empty subset $\theta = \{i_1,\dots ,i_k\} \subset \{1,2, \dots, d-1\}$ with $1 \le i_1 < i_2 < \dots < i_k \le d-1$, we define the subspace
\[
\mathfrak{a}_{\theta} := \bigcap_{j \notin \theta} \ker \alpha_j \subset \mathfrak{a},
\]
and define the projection $\pi_\theta : \mathfrak{a} \to \mathfrak{a}_\theta$ by the condition
\[
\omega_{i_s}(\pi_\theta(a)) = \omega_{i_s}(a) \quad \text{for all } i_s \in \theta \text{ and } a \in \mathfrak{a}.
\]
We identify the dual space $\mathfrak{a}_\theta^* \subset \mathfrak{a}^*$ with the span of the fundamental weights $\omega_{i_s}$ for $i_s \in \theta$:
$$
\mathfrak{a}_\theta^* := \operatorname{span}_{i_s \in \theta}(\omega_{i_s}) \subset \mathfrak{a}^*.
$$

\begin{definition}[Partial Iwasawa cocycle]\label{iwasawa}
For each $\theta$, the \emph{partial Iwasawa cocycle} is the map
\[
B_\theta : \mathsf{PGL}(d, \mathbb{R}) \times \mathcal{F}_\theta \to \mathfrak{a}_\theta
\]
determined by the condition
\[
\omega_{i_s}(B_\theta(g,F)) = \log \frac{\|g v_{i_s}\|}{\|v_{i_s}\|} \quad \text{for all } g\in \mathsf{PGL}(d,\mathbb{R}),\; F\in \mathcal{F}_\theta, \text{ and } i_s\in \theta,
\]
where $v_{i_s}$ is a non-zero vector in $\bigwedge^{i_s}(F^{(i_s)})$ (unique up to a scalar multiple), and $\|\cdot\|$ denotes the norm on $\bigwedge^{i_s}\mathbb{R}^d$ induced by the standard Euclidean norm on $\mathbb{R}^d$.
\end{definition}

Finally, $B_\theta$ satisfies the cocycle identity:
\[
B_\theta(gh, F) = B_\theta(g, hF) + B_\theta(h, F).
\]
For further details, see for example Quint~\cite{quintoriginal} or Pozzetti--Sambarino--Wienhard~\cite[Subsection 4.6]{Liplim}.

\vspace{3mm}
Let $\Gamma\subset \mathsf{PGL}(d,\mathbb{R})$ be a discrete subgroup and $\phi \in \mathfrak{a}^*$. The \emph{critical exponent} of $\Gamma$ with respect to $\phi$ is defined by
$$
\delta^{\phi}(\Gamma) := \inf \left\{ s \in [0, \infty) \;\middle|\; \mathcal{Q}_{\Gamma}^{\phi}(s)  < \infty \right\},
$$
where the $\phi$-Poincare series is defined by
\[
\mathcal{Q}_{\Gamma}^{\phi}(s) = \sum_{\gamma \in \Gamma} e^{-s\, \phi(\kappa(\gamma))}.
\]

\begin{definition}[Patterson--Sullivan measure]
Fix a non-empty subset $\theta\subset \{1,\dots,d-1\}$ and let $\Gamma\subset \mathsf{PGL}(d,\mathbb{R})$ be a $\theta$-divergent discrete subgroup. Let $\phi \in \mathfrak{a}_\theta^*$ be such that $\delta^{\phi}(\Gamma) < \infty$. A \emph{$\phi$-Patterson--Sullivan measure} for $\Gamma$ is a regular Borel probability measure $\mu$ supported on $\Lambda^{\theta}(\Gamma)$ satisfying the transformation rule
$$
\frac{d\gamma_* \mu}{d\mu}(\xi) = e^{-\delta^{\phi}(\Gamma) \phi\left(B_\theta(\gamma^{-1}, \xi)\right)}
$$
for $\mu$-almost every $\xi \in \Lambda^\theta(\Gamma)$ and every $\gamma \in \Gamma$, where $\gamma_* \mu$ denotes the pushforward of $\mu$ under the action of $\gamma$.
\end{definition}

A Patterson--Sullivan measure always exists for projective Anosov subgroups.

\begin{theorem}[{Sambarino~\cite[Section~5.5]{sambarinopattersonsullivan}}]\label{theoremofexistenceofpattersonsullivan}
Let $\theta = \{1,d-1\}$ and let $\Gamma\subset \mathsf{PGL}(d,\mathbb{R})$ be a projective Anosov subgroup. Then for any $\phi\in \mathfrak{a}_\theta^*$ that evaluates positively on $\mathfrak{a}^{+}$, we have:
\begin{enumerate}
    \item $\delta^{\phi}(\Gamma)<\infty$.
    \item The Poincar\'e series $\sum_{\gamma \in \Gamma} e^{-s\, \phi(\kappa(\gamma))}$ diverges at the critical exponent $s = \delta^{\phi}(\Gamma)$.
    \item There exists a $\phi$-Patterson--Sullivan measure on $\Lambda^\theta(\Gamma)$. 
\end{enumerate}
\end{theorem}

We have the shadow lemma for Patterson-Sullivan measures.
\begin{theorem}[{Canary--Zhang--Zimmer~\cite[Proposition~7.1]{CaZhZi3} and Sambarino~\cite[Lemma~5.7.1]{sambarinopattersonsullivan}}]\label{shadowlemma}
    Let $\Gamma_0$ be a regular convex cocompact group with respect to $(V,\Omega,\mathcal{C})$, and let $\rho:\Gamma_0\to \mathsf{PGL}(d,\mathbb{R})$ be a projective Anosov representation with limit map $\xi$.
    With the notation of Theorem~\ref{theoremofexistenceofpattersonsullivan}, fix such a functional $\phi$ and let $\mu$ be the corresponding $\phi$-Patterson--Sullivan measure. For any fixed $b_0\in \mathcal{C}$ and sufficiently large $R>0$, there exists a constant $C>1$ such that for any $\gamma\in \Gamma_0$,
    $$
    \frac{1}{C} e^{-\delta^{\phi}(\Gamma) \phi(\kappa(\rho(\gamma)))}\le \mu\big(\xi(\tilde{\mathcal{O}}_R(b_0,\gamma b_0))\big)\le C e^{-\delta^{\phi}(\Gamma) \phi(\kappa(\rho(\gamma)))}.
    $$
\end{theorem}
\begin{remark}\label{remarkofpushforwardshadowlemma}
By considering the pushforward of $\mu$ from $\Lambda^{1,d-1}(\Gamma)$ to $\Lambda^1(\Gamma)$ under the natural projection (which is denoted $\mu^1$), and letting $\xi^{1}: \Lambda^1(\Gamma_0)\to \Lambda^1(\Gamma)$ be the induced equivariant map, the inequality can be equivalently written as:
    $$
    \frac{1}{C} e^{-\delta^{\phi}(\Gamma) \phi(\kappa(\rho(\gamma)))}\le \mu^1\big(\xi^1(\hat{\mathcal{O}}_R(b_0,\gamma b_0))\big)\le C e^{-\delta^{\phi}(\Gamma) \phi(\kappa(\rho(\gamma)))}.
    $$
\end{remark}

\vspace{3mm}

For later use, we record some necessary lemmas. Under the conventions of Theorem~\ref{theoremofexistenceofpattersonsullivan} and Lemma~\ref{shadowlemma}, for any linear form $\phi$ that evaluates positively on $\mathfrak{a}^{+}$ and satisfies $\delta^{\phi}(\rho(\Gamma_0))<\infty$, we define $\mathcal{A}_n(\phi)$ as
\[
\mathcal{A}_n(\phi) := \left\{ \gamma\in \Gamma_0 \;\middle|\; e^{-\phi(\kappa(\rho(\gamma)))} \in (e^{-(n+1)}, e^{-n}] \right\},
\]
which is a finite set for each $n\ge 0$.

We recall the following separation lemma:

\begin{lemma}[{Canary--Zhang--Zimmer~\cite[Lemma~8.5]{CaZhZi}}]\label{lemmaofdisjointshadows}
Fix $b_0\in \mathcal{C}$. For any $R>0$, there exists a constant $C_0 = C_0(R)>0$ such that if $n\ge 0$, $\gamma_1,\gamma_2\in \mathcal{A}_n(\alpha_1)$, and $d_{\Omega}(\gamma_1 b_0,\gamma_2 b_0) > C_0$, then 
\[
\mathcal{O}_R(b_0,\gamma_1 b_0)\cap \mathcal{O}_R(b_0,\gamma_2 b_0) = \emptyset.
\]
\end{lemma}

We also recall the following divergence lemma:

\begin{lemma}[{Canary--Zhang--Zimmer~\cite[Lemma~8.7]{CaZhZi}}]\label{lemmaofdivergence}
For any $0<\delta<\delta^{\phi}(\rho(\Gamma_0))$, we have
\[
\limsup_{n\to \infty} \sum_{\gamma\in \mathcal{A}_n(\phi)} e^{-\delta\phi(\kappa(\rho(\gamma)))} = \infty.
\]
\end{lemma}

\subsection{The Bowen--Margulis--Sullivan Measure}\label{gromovandbms}

To address the fact that the Patterson--Sullivan measure is only quasi-invariant under the action of $\Gamma$, we need to construct invariant measures. The fundamental tool required for this construction is the Gromov product, which we introduce below.

Fix a symmetric index set $\theta \subset \{1,\dots,d-1\}$. Consider the open subset $\mathcal{TF}_{\theta} \subset \mathcal{F}_{\theta} \times \mathcal{F}_{\theta}$ formed by all transverse pairs of flags. We define a continuous function
$$
[-,-]_\theta : \mathcal{TF}_{\theta} \to \mathfrak{a}_\theta,
$$
referred to as the \emph{Gromov product}. This map is determined by the requirements:
$$
\omega_{s}([F_1,F_2]_\theta) = \log \frac{\|\bigwedge_{i=1}^{s} e_{1,i} \wedge \bigwedge_{j=1}^{d-s} e_{2,j}\|}{\|\bigwedge_{i=1}^{s} e_{1,i}\|\, \|\bigwedge_{j=1}^{d-s} e_{2,j}\|}, \qquad s \in \theta,
$$
where $\{e_{1,i}\}_{i=1}^{s}$ serves as a basis for $F_{1}^{(s)}$ and $\{e_{2,j}\}_{j=1}^{d-s}$ as a basis for $F_{2}^{(d-s)}$. Here, $\|\cdot\|$ represents the standard exterior power norm induced by the Euclidean inner product. This evaluation is independent of the specific choice of basis vectors $e_{1,i}$ and $e_{2,j}$. Additionally, the inequality $\omega_{s}([F_1,F_2]_\theta) \le 0$ holds universally.

A direct calculation shows that the Gromov product obeys the following cocycle relation:
$$
[g F_1, g F_2]_\theta - [F_1, F_2]_\theta = - B_\theta(g, F_1) - i|_{\mathfrak{a}_\theta} \circ B_\theta(g, F_2),
$$
which is valid for every $g \in \mathsf{PGL}(d,\mathbb{R})$ and any pair $(F_1, F_2) \in \mathcal{TF}_\theta$. In this identity, $i$ stands for the opposition involution acting on $\mathfrak{a}$, explicitly given by
$$
i\left(\mathrm{diag}(\lambda_1, \lambda_2, \dots, \lambda_d)\right) = \mathrm{diag}(-\lambda_d, -\lambda_{d-1}, \dots, -\lambda_1).
$$
Because $\theta$ is symmetric, the involution $i$ naturally restricts to the subspace $\mathfrak{a}_\theta$.

We now proceed to construct the invariant measures linked to a projective Anosov subgroup $\Gamma \subset \mathsf{PGL}(d,\mathbb{R})$. For the remaining of this Section, we assume that $\theta = \{1,d-1\}$. Let $\phi$ and $\mu$ represent the linear functional and the Patterson--Sullivan measure introduced in \rev{Theorem}~\ref{theoremofexistenceofpattersonsullivan}. We define $\bar{\phi} := \phi\circ i$, accompanied by its corresponding measure $\bar\mu$.

Recall the standard Radon--Nikodym transformation rule for $\mu$, which reads
$$
\frac{d(\gamma_* \mu)}{d\mu}(\xi) = e^{- \delta^\phi(\Gamma) \phi\left(B_\theta(\gamma^{-1}, \xi)\right)}.
$$

Define $\tilde{\mathcal{G}} := \Lambda^\theta(\Gamma)^{(2)}$ as the collection of distinct point pairs in the limit set. We refer to $\tilde{\mathcal{G}} \times \mathbb{R}$ as the \emph{Gromov flow space}. This product space admits an action by $\Gamma \times \mathbb{R}$, where $\Gamma$ acts \rev{by $\gamma \cdot (\xi, \eta, t) = (\gamma\xi, \gamma\eta, t + \beta(\gamma, \xi))$, here $\beta$ is the Busemann cocycle with respect to the Hilbert metric}, while $\mathbb{R}$ acts by translation on the second coordinate.

We equip $\tilde{\mathcal{G}} \times \mathbb{R}$ with a measure $\tilde{m}_\phi$ prescribed by
$$
d\tilde{m}_\phi(\xi, \eta, t) := e^{-\delta^\phi(\Gamma) \phi([\xi, \eta]_\theta)} \, d\mu(\xi) \, d\bar\mu(\eta) \, dt.
$$

By combining the Gromov product's cocycle property with the Radon--Nikodym derivative of $\mu$, we see that $\tilde{m}_\phi$ remains invariant under the \rev{above} $\Gamma$-action. Its invariance under $\mathbb{R}$-translations is also immediate from the definition.

As a result, $\tilde{m}_\phi$ naturally induces a flow-invariant Borel measure $m_\phi$ on the quotient space $(\tilde{\mathcal{G}} \times \mathbb{R}) / \Gamma$. We term $m_\phi$ the \emph{Bowen--Margulis--Sullivan measure} associated with the functional $\phi$.

Suppose there exists a regular convex cocompact subgroup $\Gamma_0$ with respect to $(V,\Omega,\mathcal{C})$ and a projective Anosov representation $\rho$ with limit map $\xi$ such that $\rho(\Gamma_0) = \Gamma$. This data provides an equivariant homeomorphism identifying $(\tilde{\mathcal{G}} \times \mathbb{R})$ with \rev{the subset $T^1\mathcal{C}$ of} the unit tangent bundle \rev{$T^1\Omega$} associated with the Hilbert metric \rev{consisting of vectors tangent to projective line geodesics with both endpoints in $\Lambda^1(\Gamma_0)$}. We let $\tilde{m}_{\phi,\Omega}$ and $m_{\phi,\Omega}$ denote the measures on $T^1\mathcal{C}$ and $T^1\mathcal{C}/\rev{\Gamma_0}$ naturally induced by this identification.

Thanks to the transversality property inherent to the limit sets of Anosov subgroups, the Gromov product $[\xi, \eta]_\theta$ is globally well-defined and continuous across $\tilde{\mathcal{G}}$. This continuity guarantees that $m_\phi$ is a locally finite measure.

Finally, we state the corresponding ergodicity results:

\begin{theorem}[{Sambarino~\cite{sambarinopattersonsullivan} and Canary--Zhang--Zimmer~\cite[Theorem~11.1]{CaZhZi3}}]\label{ergodicities}
Let $\theta = \{1,d-1\}$. Assume $\Gamma \subset \mathsf{PGL}(d,\mathbb{R})$ is a \rev{non-elementary} projective Anosov subgroup, and consider a $\phi$-Patterson--Sullivan measure $\mu$ on $\Lambda^\theta(\Gamma)$ (with $\bar\phi$ and $\bar\mu$ defined analogously) as described in Theorem~\ref{theoremofexistenceofpattersonsullivan}. We then have the following:
\begin{enumerate}
    \item The diagonal $\Gamma$-action on $\Lambda^\theta(\Gamma)^{(2)}$ is ergodic relative to the measure class of $\mu \times \bar{\mu}$. As a corollary, the $\Gamma$-action on $\Lambda^\theta(\Gamma)$ is ergodic with respect to $\mu$.
    \item   The translation flow on the quotient space $(\tilde{\mathcal{G}} \times \mathbb{R}) / \Gamma$ is ergodic with respect to the Bowen--Margulis--Sullivan measure $m_\phi$ (and correspondingly, the geodesic flow on $\rev{T^1}\mathcal{C}/\Gamma_0$ is ergodic with respect to $m_{\phi, \Omega}$).

\end{enumerate}
\end{theorem}

\section{Full dimensional limit set}
In this Section we prove the following theorem.

\begin{theorem}\label{theoremoffulldimension}
If $\Gamma\subset \mathsf{PGL}(d,\mathbb{R})$ is a projective Anosov subgroup and the Hausdorff dimension of its limit set $\Lambda^1(\Gamma)$ $\dim_H\Lambda^1(\Gamma)$ equals $d-1$, then $d = 2$ and $\Gamma$ is a cocompact lattice in $\mathsf{PGL}(2,\mathbb{R})$.
\end{theorem}

We outline the proof of this theorem. First, we show that if $\dim_H \Lambda^1(\Gamma) = d-1$, then the affinity exponent $\alpha(\Gamma) = d-1$. Next, we show that this implies $\Lambda^1(\Gamma)$ has positive measure with respect to $m_{vol}$. Finally, we apply the Lebesgue density theorem to show that the limit set is convex, which concludes the proof.

Throughout this section, we assume that $\Gamma = \rho(\Gamma_0)$, where $\Gamma_0$ is a regular convex cocompact subgroup with respect to $(V,\Omega, \mathcal{C})$ as described in Theorem~\ref{theoremofrealization}, and we retain the notations $\rho, \xi, \xi^1, \xi^{d-1}$ therein.

\subsection{Upper bound of $\alpha(\Gamma)$.}
For the first step, we prove:
\begin{proposition}\label{propositionofexponentupperbound}
If $\Gamma\subset \mathsf{PGL}(d,\mathbb{R})$ is projective Anosov, then $\alpha(\Gamma)\le d-1$.
\end{proposition}

Let $\theta = \{1,d-1\}$ and let $\phi_d = d\omega_1\in \mathfrak{a}_{\theta}^*$. We prove a lemma first:

\begin{lemma}\label{criticalexponentdomination}
Let $\Gamma\subset \mathsf{PGL}(d,\mathbb{R})$ be a projective Anosov subgroup. If the affinity exponent $\alpha(\Gamma)>d-1$, then the critical exponent $\delta^{\phi_d}(\Gamma)>1$.
\end{lemma}

\begin{proof}
Note that for any $\gamma\in \Gamma$,
$$
\log \frac{\sigma_1}{\sigma_d}(\gamma)\ge \log \frac{\sigma_1}{\sigma_2}(\gamma).
$$
From Theorem~\ref{theoremofcriterionofsingularvaluegaps}, there exist $A'>1$ and $a'>0$ such that
$$
\frac{1}{A'} d_{word}(\mathrm{id},\gamma) - a'\le \log \frac{\sigma_1(\gamma)}{\sigma_d(\gamma)}.
$$
On the other hand, for any $g,h\in \mathsf{GL}(d,\mathbb{R})$, the definition of singular values yields
$$
\sigma_1(gh)\le \sigma_1(g)\sigma_1(h).
$$
Since $\sigma_d(g) = \frac{1}{\sigma_1(g^{-1})}$, we also have
$$
\sigma_d(gh)\ge \sigma_d(g)\sigma_d(h).
$$
Therefore, $\frac{\sigma_1(gh)}{\sigma_d(gh)}\le \frac{\sigma_1(g)}{\sigma_d(g)} \frac{\sigma_1(h)}{\sigma_d(h)}$. Enlarging $A',a'$ if necessary, the subadditivity gives us
$$
\frac{1}{A'} d_{word}(\mathrm{id},\gamma) - a'\le \log \frac{\sigma_1(\gamma)}{\sigma_d(\gamma)}\le  A' d_{word}(\mathrm{id},\gamma) + a'.
$$
Note also that we can express the first fundamental weight $\omega_1$ as a positive combination of simple roots, i.e.,
$$
\omega_1 = \frac{1}{d} \sum_{i=1}^{d-1} (d-i) \alpha_i,
$$
which implies that $\log \sigma_1(g)\ge \frac{d-1}{d} \log \frac{\sigma_1(g)}{\sigma_2(g)}$. Combine this with the subadditivity of $\sigma_1$, we further enlarging $A',a'$ if necessary to obtain
$$
\frac{1}{A'} d_{word}(\mathrm{id},\gamma) - a'\le \log\sigma_1(\gamma)\le  A' d_{word}(\mathrm{id},\gamma) + a'.
$$
Thus, there exist constants $C>1, c'>0$ such that for any $\gamma\in \Gamma$,
\begin{equation}\label{equationoffuncationaldomination}
\frac{1}{C}\log \sigma_1(\gamma)-c' \le  \log \frac{\sigma_1(\gamma)}{\sigma_d(\gamma)}\le  C \log \sigma_1(\gamma) +c'.
\end{equation}

For $s \ge d-1$, the affine series has the expression of:
$$
\Phi^{\mathrm{Aff}}_{\Gamma}(s) = \sum_{\gamma \in \Gamma} \left( \frac{\sigma_2(\gamma) \cdots \sigma_{d-1}(\gamma)}{\sigma_1(\gamma)^{d-2}} \right) \left(\frac{\sigma_d(\gamma)}{\sigma_1(\gamma)}\right)^{s-(d-2)}.
$$
Since $\prod_{i=1}^d \sigma_i(\gamma) = 1$, we have $\sigma_2 \cdots \sigma_{d-1} = \frac{1}{\sigma_1 \sigma_d}$. Substituting this simplifies the sum to:
$$
\Phi^{\mathrm{Aff}}_{\Gamma}(s) = \sum_{\gamma \in \Gamma} \frac{1}{\sigma_1(\gamma)^d} \left(\frac{\sigma_d(\gamma)}{\sigma_1(\gamma)}\right)^{s-(d-1)} = \sum_{\gamma \in \Gamma} e^{-\phi_d(\kappa(\gamma))}\left(\frac{\sigma_d(\gamma)}{\sigma_1(\gamma)}\right)^{s-(d-1)}. 
$$

From Equation~(\ref{equationoffuncationaldomination}), multiplying by $-(s-(d-1)) \le 0$ yields
$$
-C(s-(d-1)) \log \sigma_1(\gamma) - c'(s-(d-1)) \le \log \left(\frac{\sigma_d(\gamma)}{\sigma_1(\gamma)}\right)^{s-(d-1)} \le -\frac{1}{C}(s-(d-1)) \log \sigma_1(\gamma) + c'(s-(d-1)).
$$
Exponentiating and multiplying by $e^{-\phi_d(\kappa(\gamma))} = \sigma_1(\gamma)^{-d}$, we obtain bounds for the summands. Recall that $\mathcal{Q}_\Gamma^{\phi_d}(s) := \sum_{\gamma \in \Gamma} e^{-s \phi_d(\kappa(\gamma))}$ denotes the Poincaré series. Summing over $\Gamma$, we get for $s\ge d-1$:
$$
e^{-c'(s-(d-1))}  \mathcal{Q}_{\Gamma}^{\phi_d}\left(1+\frac{C(s-d+1)}{d}\right)\le  
 \Phi_{\Gamma}^{\mathrm{Aff}}(s) \le  e^{c'(s-(d-1))}  \mathcal{Q}_{\Gamma}^{\phi_d}\left(1+\frac{s-d+1}{Cd}\right).
$$

By definition, if the affinity exponent $\alpha(\Gamma)>d-1$, the affine series diverges at $s = d-1+\epsilon$ for some $\epsilon>0$. So $\mathcal{Q}^{\phi_d}_{\Gamma}$ diverges at $s' = 1+\frac{\epsilon}{Cd}$, concluding the proof.
\end{proof}

\begin{proof}[{Proof of Proposition~\ref{propositionofexponentupperbound}}]
We proceed by contradiction. By Lemma~\ref{criticalexponentdomination}, it suffices to prove that $\delta^{\phi_d}(\Gamma) > 1$ can not be true. 

Fix $b_0\in \mathcal{C}$. Choose $R > 0$ large enough to satisfy the statement Theorem~\ref{shadowlemma} , and let $\mu$ be the $\phi_d$-Patterson--Sullivan measure pushed forward to $\Lambda^1(\Gamma)$. By Lemma~\ref{lemmaofellipseandshadows}~(2), there exists a sufficiently large constant $C>1$ such that for $n \gg 1$ and any $\gamma\in \mathcal{A}_n(\phi_d)$, we have
\begin{equation}\label{eqoftwoshadows}
   \frac{1}{C} e^{-\delta^{\phi_d}(\Gamma) \phi_d (\kappa(\rho(\gamma)))}\le  \mu\big(\xi^1( \hat{\mathcal{O}}_R(b_0,\gamma b_0))\big)\le \mu\big(\xi^1( \hat{\mathcal{O}}_{5R}(b_0,\gamma b_0))\big)\le C e^{-\delta^{\phi_d}(\Gamma) \phi_d (\kappa(\rho(\gamma)))}, 
\end{equation}
and the geometric inclusion
\[
        \mathcal{E}\left(\rho(\gamma),\frac{1}{C}\right)\cap \Lambda^1(\rho(\Gamma_0))\subset \xi^1( \hat{\mathcal{O}}_R(b_0,\gamma b_0))\subset \xi^1( \hat{\mathcal{O}}_{5R}(b_0,\gamma b_0))\subset  \mathcal{E}(\rho(\gamma), C).
\]

For fixed $C>0$, recall that for any $1$-divergent $g\in \mathsf{PGL}(d,\mathbb{R})$, the volume of its corresponding ellipsoid $\mathcal{E}(g, C)$ is governed up to constant multiple by
\[
e^{-\phi_d(\kappa(g))} = \left( \frac{1}{\sigma_1(g)}\right)^d = \frac{\sigma_2(g)}{\sigma_1(g)}\cdots \frac{\sigma_d(g)}{\sigma_1(g)}.
\]
This implies that for the fixed $R$ and $C$, there exists a constant $C'>1$ such that for $n\gg 1$ and $\gamma \in \mathcal{A}_n(\phi_d)$, we have
\[
 \frac{1}{C'}e^{-\phi_d(\kappa(\rho(\gamma)))}\le m_{vol}\left(\mathcal{E}\left(\rho(\gamma), \frac{1}{C}\right)\right) \le m_{vol}\big(\mathcal{E}(\rho(\gamma), C)\big)\le C'e^{-\phi_d(\kappa(\rho(\gamma)))}.
\]
\begin{revblock}
Let $F_0 \subset \Gamma_0$ and $K = K(C) \ge 2$ be given by Proposition~\ref{propositionofellipsevitali}. By Lemma~\ref{lemmaofellipseandshadows}~(1) and Theorem~\ref{shadowlemma}, there exist $R' > 0$ and $C_1 > 1$ such that $\mathcal{E}(\rho(\gamma), KC) \cap \Lambda^1(\Gamma) \subset \xi^1(\hat{\mathcal{O}}_{R'}(b_0,\gamma b_0))$ and $\mu\big(\xi^1(\hat{\mathcal{O}}_{R'}(b_0,\gamma b_0))\big) \le C_1 e^{-\delta^{\phi_d}(\Gamma)\phi_d(\kappa(\rho(\gamma)))}$ for all $\gamma \in \Gamma_0 \setminus \{\mathrm{id}\}$. Together with the volume estimate above, there exist $C_2 > 1$ and $N > 1$ (so that every $\gamma \in \mathcal{A}_n(\phi_d)$ with $n > N$ lies in $\Gamma_0 \setminus F_0$) such that if $n > N$ and $\gamma \in \mathcal{A}_n(\phi_d)$, we have:
\[
\mu\big(\mathcal{E}(\rho(\gamma), KC) \cap \Lambda^1(\Gamma)\big) \le C_2 \big(m_{vol}(\mathcal{E}(\rho(\gamma), C))\big)^{\delta^{\phi_d}(\Gamma)}.
\]
\end{revblock}

For any $\epsilon>0$, let $N_{\epsilon}>N$ be chosen such that if $n>N_{\epsilon}$ and $\gamma\in \mathcal{A}_n(\phi_d)$, then \rev{$m_{vol}\big(\mathcal{E}(\rho(\gamma), C)\big)<\epsilon$}. Let $\mathcal{I}_{\epsilon} = \{\gamma\in \Gamma_0 \mid \gamma \in \mathcal{A}_n(\phi_d) \text{ for some } n>N_{\epsilon}\}$.

Note that the union $\bigcup_{\gamma\in \mathcal{I}_{\epsilon}} \xi^1(\hat{\mathcal{O}}_R(b_0,\gamma b_0))$ \rev{(and hence $\bigcup_{\gamma\in \mathcal{I}_{\epsilon}} \mathcal{E}(\rho(\gamma), C)$)} covers $\Lambda^1(\Gamma)$. From \rev{Proposition~\ref{propositionofellipsevitali}}, we can extract a subset $\mathcal{J}_{\epsilon}\subset \mathcal{I}_{\epsilon}$ such that the family \rev{$\{\mathcal{E}(\rho(\gamma), C) \mid \gamma \in \mathcal{J}_{\epsilon}\}$ is pairwise disjoint in $\mathbb{RP}^{d-1}$, and}
\begin{revblock}
\[
    \bigcup_{\gamma \in \mathcal{I}_{\epsilon}} \mathcal{E}(\rho(\gamma), C) \subset \bigcup_{\gamma \in \mathcal{J}_{\epsilon}} \mathcal{E}(\rho(\gamma), KC).
\]

We can now estimate the total measure of the limit set:
\begin{align*}
\mu(\Lambda^1(\Gamma)) &\le \mu\left( \bigcup_{\gamma \in \mathcal{J}_{\epsilon}} \mathcal{E}(\rho(\gamma), KC) \cap \Lambda^1(\Gamma) \right) \le \sum_{\gamma\in \mathcal{J}_{\epsilon}}\mu\big(\mathcal{E}(\rho(\gamma), KC) \cap \Lambda^1(\Gamma)\big) \\
&\le C_2 \sum_{\gamma\in \mathcal{J}_{\epsilon}} \big(m_{vol}(\mathcal{E}(\rho(\gamma), C))\big)^{\delta^{\phi_d}(\Gamma)} \\
&= C_2 \sum_{\gamma\in \mathcal{J}_{\epsilon}} m_{vol}\big(\mathcal{E}(\rho(\gamma), C)\big) \cdot \big(m_{vol}(\mathcal{E}(\rho(\gamma), C))\big)^{\delta^{\phi_d}(\Gamma)-1} \\
&< C_2 \epsilon^{\delta^{\phi_d}(\Gamma)-1} \sum_{\gamma\in \mathcal{J}_{\epsilon}} m_{vol}\big(\mathcal{E}(\rho(\gamma), C)\big) \\
&= C_2 \epsilon^{\delta^{\phi_d}(\Gamma)-1}  m_{vol}\left(\bigcup_{\gamma\in \mathcal{J}_{\epsilon}}\mathcal{E}(\rho(\gamma), C)\right) \\
&\le C_2 \epsilon^{\delta^{\phi_d}(\Gamma)-1} m_{vol}(\mathbb{RP}^{d-1}).
\end{align*}
\end{revblock}
Since $\delta^{\phi_d}(\Gamma) > 1$ and $\epsilon > 0$ can be chosen to be arbitrarily small, we have $\mu(\Lambda^1(\Gamma)) = 0$, which contradicts the fact that $\mu$ is a probability measure.
\end{proof}

\subsection{Positive Volume of the Limit Set}

Recall that $m_{vol}$ denotes the volume measure on $\mathbb{RP}^{d-1}$ introduced in Section~\ref{subsectionofvolume}. In this subsection, we prove the following proposition:

\begin{proposition}\label{propositionofwhenupperboundisattained}
If $\Gamma\subset \mathsf{PGL}(d,\mathbb{R})$ is a projective Anosov subgroup, then $\dim_H \Lambda^1(\Gamma) = d -1 $ implies that $m_{vol}(\Lambda^1(\Gamma))>0$.
\end{proposition}

\begin{proof}
    The idea of the proof is similar to that of the previous measure estimates.

    Theorem~\ref{pswtheomrem} implies that if $\dim_H \Lambda^1(\Gamma) = d -1$, then $\alpha(\Gamma)\ge d-1$. On the other hand, Proposition~\ref{propositionofexponentupperbound} guarantees that $\alpha(\Gamma)\le d-1$. Thus, we must have exactly $\alpha(\Gamma)= d-1$. As shown in the proof of Lemma~\ref{criticalexponentdomination}, this implies that the critical exponent for the functional $\phi_d = d\omega_1$ is precisely $\delta^{\phi_d}(\Gamma) = 1$.

    Let $\mu$ be the associated $\phi_d$-Patterson--Sullivan measure pushed forward to $\Lambda^1(\Gamma)$.

    Fix $b_0\in \mathcal{C}$. Following the proof of Proposition~\ref{propositionofexponentupperbound} \rev{(with $\delta^{\phi_d}(\Gamma) = 1$)}, for any sufficiently large $R>0$, there \rev{exist constants $C>1$, $K = K(C) \ge 2$, $C_2 > 1$, and a finite subset $F \subset \Gamma_0$ (containing the finite set $F_0$ from Proposition~\ref{propositionofellipsevitali}) such that for all $\gamma \in \Gamma_0 \setminus F$}, we have:
\begin{revblock}
\[
\xi^1\big(\hat{\mathcal{O}}_R(b_0,\gamma b_0)\big) \subset \mathcal{E}(\rho(\gamma), C)
\quad \text{and} \quad
\mu\big(\mathcal{E}(\rho(\gamma), KC) \cap \Lambda^1(\Gamma)\big) \le C_2 \, m_{vol}\big(\mathcal{E}(\rho(\gamma), C)\big).
\]
\end{revblock}

    Since $\Gamma_0$ acts cocompactly on $\mathcal{C}$, we can further enlarge $R>0$ so that for any $x\in \Lambda^1(\Gamma_0)$, there exists a sequence $(\gamma_n)_{n\in \mathbb{N}}$ in $\Gamma_0$ such that $\gamma_n b_0\to x$ and $x\in \hat{\mathcal{O}}_R(b_0,\gamma_n b_0)$ for all $n$. 

        Now, suppose for the sake of contradiction that $m_{vol}(\Lambda^1(\Gamma)) = 0$. By the outer regularity of the Lebesgue measure, for any $\epsilon>0$, we can find an open neighborhood $U_{\epsilon} \subset \mathbb{RP}^{d-1}$ containing $\Lambda^1(\Gamma)$ such that $m_{vol}(U_{\epsilon}) < \epsilon$. By \rev{Lemma~\ref{lemmaofcartanproperty}~(1) (passing to a subsequence if necessary), the ellipsoids $\mathcal{E}(\rho(\gamma_n), C)$} shrink to the point $\xi^1(x)$ as $n \to \infty$. Thus, for the fixed $x$ and corresponding sequence $(\gamma_n)$, there exists $N>0$ such that $n>N$ implies \rev{$\gamma_n \in \Gamma_0 \setminus F$ and $\mathcal{E}(\rho(\gamma_n), C) \subset U_{\epsilon}$}.

    Let \rev{$\mathcal{I}'_{\epsilon} = \{\gamma \in \Gamma_0 \setminus F \mid \mathcal{E}(\rho(\gamma), C) \subset U_{\epsilon}\}$}. Note that the union \rev{$\bigcup_{\gamma\in \mathcal{I}'_{\epsilon}} \mathcal{E}(\rho(\gamma), C)$ covers $\Lambda^1(\Gamma)$}. From \rev{Proposition~\ref{propositionofellipsevitali}}, we can extract a subset $\mathcal{J}'_{\epsilon}\subset \mathcal{I}'_{\epsilon}$ such that the family \rev{$\{\mathcal{E}(\rho(\gamma), C) \mid \gamma \in \mathcal{J}'_{\epsilon}\}$ is pairwise disjoint in $\mathbb{RP}^{d-1}$, and}
\begin{revblock}
\[
    \bigcup_{\gamma \in \mathcal{I}'_{\epsilon}} \mathcal{E}(\rho(\gamma), C) \subset \bigcup_{\gamma \in \mathcal{J}'_{\epsilon}} \mathcal{E}(\rho(\gamma), KC).
\]
\end{revblock}

We have the following estimate as we did in the proof of Proposition~\ref{propositionofexponentupperbound}:
\begin{revblock}
\begin{align*}
\mu(\Lambda^1(\Gamma)) &\le \mu\left(\bigcup_{\gamma \in \mathcal{J}'_{\epsilon}} \mathcal{E}(\rho(\gamma), KC) \cap \Lambda^1(\Gamma)\right) \\
&\le \sum_{\gamma \in \mathcal{J}'_{\epsilon}} \mu\big(\mathcal{E}(\rho(\gamma), KC) \cap \Lambda^1(\Gamma)\big) \\
&\le C_2 \sum_{\gamma \in \mathcal{J}'_{\epsilon}} m_{vol}\big(\mathcal{E}(\rho(\gamma), C)\big) \\
&= C_2 \, m_{vol}\left( \bigcup_{\gamma \in \mathcal{J}'_{\epsilon}} \mathcal{E}(\rho(\gamma), C) \right) \\
&\le C_2 \, m_{vol}(U_{\epsilon})\\
&< C_2 \epsilon.
\end{align*}
\end{revblock}

Since $\epsilon > 0$ can be chosen arbitrarily small, we conclude that $\mu(\Lambda^1(\Gamma)) = 0$, which contradicts the fact that $\mu$ is a probability measure. 
\end{proof}

\subsection{Convexity of the limit set}
The last step of proving Theorem~\ref{theoremoffulldimension} is due to the following proposition:

\begin{proposition}\label{propositionofcontainingaline}
If $\Gamma\subset\mathsf{PGL}(d,\mathbb{R})$ is a projective Anosov subgroup and $\mathrm{dim}_{H} \Lambda^1(\Gamma) = d - 1$, then there exist distinct $x, y \in \Lambda^1(\Gamma)$ such that $\mathrm{Span}(x,y) \subset \Lambda^1(\Gamma)$.
\end{proposition}

\begin{proof}
From Theorem~\ref{theoremofrealization}, we may assume that $\Gamma$ is the image of an Anosov representation $\rho$ from a regular convex cocompact subgroup $\Gamma_0$ with respect to $(V,\Omega,\mathcal{C})$ with limit map $\xi$. Fix $b_0 \in \mathcal{C}$, and pick large enough constants $R>0, a>1, A>0$ satisfying the requirements of Lemma~\ref{lemmaofquasigeodesics},  and then pick $C>0$ so that $R,C$ satisfy Lemma~\ref{lemmaofellipseandshadows}~(2).

If $\mathrm{dim}_{H} \Lambda^1(\Gamma) = d - 1$, then by Proposition~\ref{propositionofwhenupperboundisattained}, there exists a Lebesgue density point $p\in \Lambda^1(\Gamma)$. Let $p' \in \Lambda^1(\Gamma_0)$ be such that $\xi^1(p') = p$. By Lemma~\ref{lemmaofquasigeodesics}, we can pick a sequence $(\gamma_n)_{n\in \mathbb{Z}^{>0}}$ representing an $(a,A)$-quasi-geodesic ray such that $\gamma_n b_0 \to p'$ as $n \to +\infty$ and
\[
[b_0,p') \subset \bigcup_{n \in \mathbb{Z}^{>0}} B_{d_{\Omega}}(\gamma_n b_0, R).
\]

By passing to a subsequence (note that this subsequence may no longer represent a quasi-geodesic), we may assume that $\gamma_n^{-1} b_0\to y\in \Lambda^1(\Gamma_0)$, $\gamma_n^{-1} p'\to x\in \Lambda^1(\Gamma_0)$, and $p'\in \hat{\mathcal{O}}_R(b_0,\gamma_n b_0)$. Note that the last condition implies that $d_{\Omega}([x,y],b_0)\le R$, and thus $x\neq y$. We will show that $\mathrm{Span}(\xi^1(x),\xi^1(y))\subset \Lambda^1(\Gamma)$.

Fix $C_0 > C$. By further passing to a subsequence, we may assume from Lemma~\ref{lemmaofellipseandshadows} that
\[
    \xi^1( \hat{\mathcal{O}}_R(b_0,\gamma_n b_0)) \subset \mathcal{E}(\rho(\gamma_n), C_0).
\]
Since $p\in \mathcal{E}(\rho(\gamma_n), C_0)$, which is the projectivization of an ellipse in the affine chart with minor axis of length $C_0 \frac{\sigma_d}{\sigma_1}(\rho(\gamma_n))$, we may further assume that for sufficiently large $n$,
\begin{equation}\label{equation1inconvexityproof}
    \mathcal{B}\left(p, C_0 \frac{\sigma_d}{\sigma_1}(\rho(\gamma_n))\right) \subset \mathcal{E}(\rho(\gamma_n), 3C_0).
\end{equation}
Since $p$ is a Lebesgue density point of $\Lambda^1(\Gamma)$, the density 
\begin{equation}\label{equationofdensityintheproofofconvexity}
D_{C_0,n} = \frac{m_{vol}\left( \mathcal{B}\left(p, C_0 \frac{\sigma_d}{\sigma_1}(\rho(\gamma_n))\right)\cap \Lambda^1(\Gamma)\right)}{m_{vol}\left( \mathcal{B}\left(p, C_0 \frac{\sigma_d}{\sigma_1}(\rho(\gamma_n))\right)\right)}    
\end{equation}
satisfies $\lim_{n\to \infty} D_{C_0,n} = 1$.

Let $\mathcal{E}_{C_0,n} = \rho(\gamma_n)^{-1} \mathcal{B}\left(p, C_0 \frac{\sigma_d}{\sigma_1}(\rho(\gamma_n))\right)$. From Lemma~\ref{lemmaofmapofellipse} and Equation~\ref{equation1inconvexityproof}, we have
\[
\mathcal{E}_{C_0,n} \subset \mathcal{B}((V_{d-1}(\rho(\gamma_n)))^{\perp}, 3C_0).
\]
Now we control the Jacobian of $\rho(\gamma_n)$ (with respect to $m_{vol}$) on this latter set. By Equation~\ref{equationofcocycleformeasure}, if $g=\rho(\gamma_n)$ and $q=[v]$, then
		\[
		\mathrm{Jac}\, g(q)=\left(\frac{\|v\|}{\|gv\|}\right)^d.
		\]
		Write $g=u(g)a(g)v(g)$. For $q\in \mathcal{B}((V_{d-1}(g))^{\perp},3C_0)$, after applying $v(g)$ we may choose a unit representative of $[v]$ whose $e_1$-coordinate is bounded from below by a positive constant depending only on $C_0$. Hence $\|gv\|$ is comparable to $\sigma_1(g)\|v\|$, up to a multiplicative constant uniform in $g$ and in $q$. Thus there exists a constant $C_1=C_1(C_0)>1$, independent of $n$, such that
		\[
		\frac{\sup_{x\in \mathcal{B}((V_{d-1}(\rho(\gamma_n)))^{\perp}, 3C_0)} \mathrm{Jac} \rho(\gamma_n)(x)}{\inf_{x\in \mathcal{B}((V_{d-1}(\rho(\gamma_n)))^{\perp}, 3C_0)} \mathrm{Jac} \rho(\gamma_n)(x)}\le C_1,
		\]


As a result, Equation~\ref{equationofdensityintheproofofconvexity} implies that the pullback density
		\begin{equation}\label{equationofpullbackdensityintheproofofconvexity}
			D'_{C_0,n} = \frac{m_{vol}\left( \mathcal{E}_{C_0,n}\cap \Lambda^1(\Gamma)\right)}{m_{vol}\left( \mathcal{E}_{C_0,n}\right)}    
		\end{equation}
satisfies $\lim_{n\to \infty} D'_{C_0,n} = 1$.

On the other hand, from the Cartan property in Lemma~\ref{lemmaofcartanproperty}, we have $U_1(\rho(\gamma_n))\to p$ and $V_{d-1}(\rho(\gamma_n))\to \rev{\xi^{d-1}(H_y)}$\rev{. By Lemma~\ref{lemmaofmapofellipse}, in the affine chart $TS^{d-1}_{V_{d-1}(\rho(\gamma_n))^{\perp}}$,} the maximum expansion rate of $\rho(\gamma_n)^{-1}$ in a neighborhood of $p$ can be estimated as $\frac{\sigma_1}{\sigma_d}(\rho(\gamma_n))$ up to a uniformly bounded \rev{multiplicative} error. Multiplying this expansion rate by the initial radius $C_0 \frac{\sigma_d}{\sigma_1}(\rho(\gamma_n))$, the length of the principal axis of $\mathcal{E}_{C_0,n}$ \rev{in this chart} is bounded from below by a constant $C_2(C_0)$ which satisfies $\lim_{C_0\to \infty} C_2(C_0) = \infty$ (and it is also bounded above by a constant $C_3(C_0)$). In any other direction, the expansion rate is proportional to $\frac{\sigma_1}{\sigma_i}(\rho(\gamma_n))$ for some \rev{$2 \le i \le d-1$}. Since the ratio of the maximum expansion rate to the expansion rates in other directions goes to infinity (as for \rev{$2 \le i \le d-1$}, $\frac{\sigma_i}{\sigma_d}(\rho(\gamma_n)) \to \infty$), the lengths of all other axes of $\mathcal{E}_{C_0,n}$ converge to 0. Therefore, \rev{for $C_0$ and $n$ sufficiently large, $\mathcal{E}_{C_0,n}$ is comparable (up to a uniform scaling factor around $\rho(\gamma_n)^{-1}(p)$) to} an ellipse \rev{centered at $\rho(\gamma_n)^{-1}(p)$} whose principal axis intersects $\partial \mathcal{B}(\rev{\xi^{d-1}(H_y)^{\perp}},\frac{C_2(C_0)}{3})$ twice and all its other axes converge to 0.

Moreover, the principal axis of \rev{this approximating} ellipsoid points to the direction of maximum expansion, i.e., $U_1(\rho(\gamma_n)^{-1})$, and note that the ellipsoid also contains $\rho(\gamma_n)^{-1}(p)$. Since $\rho(\gamma_n)^{-1}(p) = \xi^1(\gamma_n^{-1} p') \to \xi^1(x)$ and $U_1(\rho(\gamma_n)^{-1}) \to \xi^1(y)$, \rev{the projective line containing} this principal axis converges to the line $\mathrm{Span}(\xi^1(x),\xi^1(y))$.

The limit $\lim_{n\to \infty} D'_{C_0,n} = 1$ shows that $\mathcal{B}(\rev{\xi^{d-1}(H_y)^{\perp}},\frac{C_2(C_0)}{3})\cap \mathrm{Span}(\xi^1(x),\xi^1(y))\subset \Lambda^1(\Gamma)$. Indeed, if this were not the case, there would exist an open ball in $\mathcal{B}(\rev{\xi^{d-1}(H_y)^{\perp}},\frac{C_2(C_0)}{3})$ intersecting $\mathrm{Span}(\xi^1(x),\xi^1(y))$ but disjoint from $\Lambda^1(\Gamma)$. As $\mathcal{E}_{C_0,n}$ collapses to the principal axis, this open ball would eventually capture a strictly positive, uniform proportion of the volume of $\mathcal{E}_{C_0,n}$. This missing volume would make $\limsup_{n\to \infty} D'_{C_0,n}<1$, which is a contradiction. Since $C_0$ is arbitrary and $\lim_{C_0\to \infty} C_2(C_0 ) = \infty$, we conclude that $\mathrm{Span}(\xi^1(x),\xi^1(y))\subset \Lambda^1(\Gamma)$.

\end{proof}

\begin{proof}[Proof of Theorem~\ref{theoremoffulldimension}]
\begin{revblock}
By Proposition~\ref{propositionofcontainingaline}, there exist distinct $x,y \in \Lambda^1(\Gamma)$ such that the projective line $L := \mathrm{Span}(x,y)$ is contained in $\Lambda^1(\Gamma)$. For any $p \in \Lambda^1(\Gamma)$, since $\Gamma$ is projective Anosov, there exists a hyperplane $H_p \in \Lambda^{d-1}(\Gamma)$ with $(p, H_p) \in \Lambda^{1,d-1}(\Gamma)$. Because $L$ and $H_p$ are the projectivizations of subspaces of $\mathbb{R}^d$ of dimensions $2$ and $d-1$, respectively, their intersection $L \cap H_p$ is non-empty. Pick any $q \in L \cap H_p$. Since $L \subset \Lambda^1(\Gamma)$, there exists $H_q \in \Lambda^{d-1}(\Gamma)$ such that $(q, H_q) \in \Lambda^{1,d-1}(\Gamma)$. As $q \in H_p$, the flags $(q, H_q)$ and $(p, H_p)$ are not transverse; by Definition~\ref{definitionoftransverse}, this forces $(q, H_q) = (p, H_p)$, and in particular $p = q \in L$.

Since $p \in \Lambda^1(\Gamma)$ was arbitrary, we obtain $\Lambda^1(\Gamma) = L \cong \mathbb{RP}^1$. Consequently,
\[
d - 1 = \dim_H \Lambda^1(\Gamma) = \dim_H(\mathbb{RP}^1) = 1,
\]
which forces $d = 2$. Thus $\Gamma$ is a projective Anosov subgroup of $\mathsf{PGL}(2,\mathbb{R})$ with $\Lambda^1(\Gamma) = \mathbb{RP}^1$, and hence a cocompact lattice in $\mathsf{PGL}(2,\mathbb{R})$.
\end{revblock}
\end{proof}

\section{The Barbot Component}
In this section, we assume that $\Gamma_0$ is a cocompact lattice in $\mathsf{PSL}(2,\mathbb{R})$ acting on the hyperbolic disk $\mathbb{D}$, and $\rho:\Gamma_0\to \mathsf{PGL}(3,\mathbb{R})\cong \mathsf{PSL}(3,\mathbb{R})$ is a projective Anosov representation whose image is irreducible and has limit map $\xi$. Let $\Gamma = \rho(\Gamma_0)$. We are going to prove:

\begin{theorem}\label{equivalenceofdim1}
Based on the fixed notation, the following are equivalent:
\begin{enumerate}
    \item $\dim_H(\xi^1(\partial\mathbb{D})) = 1$.
    \item $\alpha(\Gamma) = 1$.
    \item $\delta^{\alpha_1}(\Gamma) = 1$.
    \item $\xi^1(\partial\mathbb{D})$ is a rectifiable curve.
    \item $\xi^1(\partial\mathbb{D})$ is a $C^1$ curve.
    \item $\rho$ is a Hitchin representation.
\end{enumerate}
\end{theorem}

Since the topological curve $\xi^1(\partial \mathbb{D})$ always has a Hausdorff dimension of at least one, we have:
\begin{corollary}\label{strictdimensionforbarbot}
Based on the fixed notation, if $\rho$ is a Barbot representation (see Barbot~\cite{barbot1}), then $\dim_H(\xi^1(\partial\mathbb{D})) > 1$.
\end{corollary}

\begin{proof}[Proof of Theorem~\ref{equivalenceofdim1}]
We will prove the theorem by showing the following cycle of implications. 

\textbf{(1) $\implies$ (2):} See Li--Pan--Xu~\cite[Theorem~1.2]{lipanxu}.

\textbf{(2) $\implies$ (3):} We are given $\alpha(\Gamma) = 1$.
First, by definition, for $s \in (0, 1]$, the affine series coincides with the $\alpha_1$-Poincar\'e series: $\Phi_{\Gamma}^{\mathrm{Aff}}(s) = \mathcal{Q}_{\Gamma}^{\alpha_1}(s)$. Since $\alpha(\Gamma) = 1$, the affine series $\Phi_{\Gamma}^{\mathrm{Aff}}(s)$ must diverge for $s<1$. This implies $\mathcal{Q}_{\Gamma}^{\alpha_1}(s)$ also diverges for $s<1$, so by definition of the critical exponent, $\delta^{\alpha_1}(\Gamma) \ge 1$.

Next, we show that $\delta^{\alpha_1}(\Gamma) > 1$ leads to a contradiction. 
Fix $b_0\in \mathbb{D}$ and let $d_{\mathbb{D}}$ denote the hyperbolic metric on $\mathbb{D}$. Since $\rho$ is Anosov, there exist constants $a>1$ and $A>0$ such that for any $\gamma \in \Gamma_0$:
\[
\frac{1}{a} d_{\mathbb{D}}(b_0,\gamma b_0)-A\le \log \frac{\sigma_1}{\sigma_2}(\rho(\gamma))\le a d_{\mathbb{D}}(b_0,\gamma b_0)+A,
\]
\[
\frac{1}{a} d_{\mathbb{D}}(b_0,\gamma b_0)-A\le \log \frac{\sigma_2}{\sigma_3}(\rho(\gamma))\le a d_{\mathbb{D}}(b_0,\gamma b_0)+A
\]
(see Canary--Zhang--Zimmer~\cite[Theorem~1.1]{CaZhZi2}).
This implies there exist constants $c>1$ and $C>0$ such that for any $\gamma \in \Gamma_0$,
\[
\log \frac{\sigma_2}{\sigma_3}(\rho(\gamma)) \le c \log \frac{\sigma_1}{\sigma_2}(\rho(\gamma)) + C,
\]
which yields the lower bound for the contraction ratio:
\[
\frac{\sigma_3}{\sigma_2}(\rho(\gamma)) \ge e^{-C} \left(\frac{\sigma_2}{\sigma_1}(\rho(\gamma))\right)^c.
\]
For $s>1$, we have
\begin{align*}
\Phi_{\Gamma}^{\mathrm{Aff}}(s) &= \sum_{\gamma \in \Gamma_0} \frac{\sigma_2}{\sigma_1}(\rho(\gamma)) \left(\frac{\sigma_3}{\sigma_1}(\rho(\gamma))\right)^{s-1} \\
&= \sum_{\gamma \in \Gamma_0} \left(\frac{\sigma_2}{\sigma_1}(\rho(\gamma))\right)^s \left(\frac{\sigma_3}{\sigma_2}(\rho(\gamma))\right)^{s-1} \\
&\ge \sum_{\gamma \in \Gamma_0} \left(\frac{\sigma_2}{\sigma_1}(\rho(\gamma))\right)^s \left( e^{-C} \left(\frac{\sigma_2}{\sigma_1}(\rho(\gamma))\right)^c \right)^{s-1} \\
&= e^{-C(s-1)} \sum_{\gamma \in \Gamma_0} \left(\frac{\sigma_2}{\sigma_1}(\rho(\gamma))\right)^{s + c(s-1)} \\
&= e^{-C(s-1)} \mathcal{Q}_{\Gamma}^{\alpha_1}\big(s + c(s-1)\big).
\end{align*}
This inequality shows that the affine series $\Phi_{\Gamma}^{\mathrm{Aff}}(s)$ diverges if the Poincar\'e series $\mathcal{Q}_{\Gamma}^{\alpha_1}\big(s + c(s-1)\big)$ diverges.

Now, assume for contradiction that $\delta^{\alpha_1}(\Gamma) > 1$.
This means $\mathcal{Q}_{\Gamma}^{\alpha_1}(s')$ diverges for all $s' < \delta^{\alpha_1}(\Gamma)$.
Thus, $\Phi_{\Gamma}^{\mathrm{Aff}}(s)$ must diverge for all $s$ satisfying $s + c(s-1) < \delta^{\alpha_1}(\Gamma)$.
This condition is equivalent to $s < \frac{\delta^{\alpha_1}(\Gamma) + c}{1+c}$.
The affinity exponent $\alpha(\Gamma)$ must therefore satisfy $\alpha(\Gamma) \ge \frac{\delta^{\alpha_1}(\Gamma) + c}{1+c}>1$. This implies $\alpha(\Gamma) > 1$, which contradicts our hypothesis $\alpha(\Gamma) = 1$.

\textbf{(3) $\implies$ (4):}  Fix $b_0\in \mathbb{D}$. For any $z\in \mathbb{D}$ and $R>0$, recall that the shadow from $b_0$ to $z$ of radius $R$ is defined as
\[
\mathcal{O}_{R}(b_0,z) := \{y\in \partial \mathbb{D} \mid [b_0,y)\cap \overline{B}_{d_{\mathbb{D}}}(z,R)\neq \emptyset\},
\]
where $[b_0,y)$ denotes the geodesic ray.

Note that for projective Anosov subgroups in $\mathsf{PGL}(3,\mathbb{R})$, we have $\mathfrak{a}_{\theta}^* = \mathfrak{a}^*$, where $\theta = \{1,2\}$, so $\alpha_1\in \mathfrak{a}_{\theta}^* $. From Theorem~\ref{theoremofexistenceofpattersonsullivan}, the $\alpha_1$ Patterson--Sullivan measure exists, and we denote its pullback on $\partial \mathbb{D}$ by $\mu$. The shadow lemma (Theorem~\ref{shadowlemma}) can be stated as: for any $r\ge 0$ and large enough $R>0$, there exists a constant $C>1$ such that for any $z\in \mathbb{D}$ and $\gamma\in \Gamma_0$, if $d_{\mathbb{D}}(z,\gamma b_0)\le r$, then
\[
\frac{1}{C}\frac{\sigma_2}{\sigma_1}(\rho(\gamma)) \le \mu (\mathcal{O}_R(b_0,z))\le C \frac{\sigma_2}{\sigma_1}(\rho(\gamma)).
\]

We also invoke the diameter estimate for shadows. For any $z\neq b_0$, let $l_{b_0}(z)$ be the intersection point between $[b_0, z)$ and $\partial \mathbb{D}$. Recall that $d_{\angle}$ is the angular metric on $\mathbb{RP}^2$.
\begin{lemma}[{Canary--Zhang--Zimmer~\cite[Theorem~5.1]{CaZhZi}}]\label{shadowdiameter}
Under the fixed notation, for any $r\ge 0$ and $R>0$, there exists a constant $C>1$ with the following property: for any $z\in \mathbb{D}$ and $\gamma\in \Gamma_0$, if $d_{\mathbb{D}}(z,\gamma b_0)\le r$, then
\[
\xi^1(\mathcal{O}_R(b_0,z))\subset B_{d_{\angle}}\left(\xi^1(l_{b_0}(z)), C \frac{\sigma_2}{\sigma_1}(\rho(\gamma))\right).
\]
\end{lemma}

Now, we prove the implication (3) $\implies$ (4). Let $d_{S^1}$ be the angular metric on $\partial \mathbb{D} = S^1$ (with respect to $b_0$) such that the diameter of $\partial \mathbb{D}$ is $\pi$. For any $x,y\in \partial \mathbb{D}$ such that $d_{S^1}(x,y)<\pi$, let $(x,y)_{arc}\subset S^1$ be the short arc connecting them, and let $m_{x,y}$ be the midpoint of this arc.

Fix a large enough $R>0$ so that the shadow lemma holds, and let $r_0$ be the diameter of the closed hyperbolic surface $\mathbb{D}/\Gamma_0$. For any such $x,y$, let $z_{x,y}$ be the unique point such that:
\begin{itemize}
    \item $z_{x,y}\in [b_0,m_{x,y})$.
    \item $B_{d_{\mathbb{D}}}(z_{x,y},R)$ is tangent to $[b_0,x)$ and $[b_0,y)$ simultaneously.
\end{itemize}
Since $r_0$ is the diameter, for any $x,y$ as above, we can find $\gamma_{x,y}\in\Gamma_0$ such that $d_{\mathbb{D}}(z_{x,y},\gamma_{x,y} b_0)\le r_0$. Applying Theorem~\ref{shadowlemma} and Lemma~\ref{shadowdiameter} (with $z=z_{x,y}$ and $\gamma=\gamma_{x,y}$), we know there exists a constant $C>1$ such that for any such $x,y$:
\[
d_{\angle}(\xi^1(x),\xi^1(y))\le C \frac{\sigma_2}{\sigma_1}(\rho(\gamma_{x,y})),
\]
\[
\frac{1}{C} \frac{\sigma_2}{\sigma_1}(\rho(\gamma_{x,y})) \le \mu(\mathcal{O}_{R}(b_0,z_{x,y})) \le C \frac{\sigma_2}{\sigma_1}(\rho(\gamma_{x,y})).
\]
Note that by our construction of $z_{x,y}$, the shadow $\mathcal{O}_{R}(b_0,z_{x,y})$ coincides with the arc $(x,y)_{arc}$. Combining the inequalities, we have
\[
d_{\angle}(\xi^1(x),\xi^1(y))\le C^2\mu((x,y)_{arc}).
\]
So, for any clockwise-ordered tuple $(x_1,\dots,x_n)$ on $\partial \mathbb{D}$ such that $d_{S^1}(x_i,x_{i+1})<\pi$ for $1\le i\le n$ (we set $x_{n+1} = x_1$), we have
\[
\sum_{i = 1}^{n} d_{\angle}(\xi^1(x_i),\xi^1(x_{i+1}))\le C^2\sum_{i = 1}^{n}\mu((x_i,x_{i+1})_{arc}) = C^2 \mu(\partial \mathbb{D}).
\]
Since $\mu(\partial \mathbb{D})=1$, the sum is uniformly bounded. Therefore, the curve is rectifiable.

\vspace{3mm}

\textbf{\rev{Arguments related to the remaining implications can be founded in Barbot~\cite[Corollary~7.2]{barbot1}, but we include a self-contained proof here for completeness.}}

\vspace{3mm}

\textbf{(4) $\implies$ (5):} If $\xi^1(\partial \mathbb{D})$ is a rectifiable curve, let $m^*$ denote the 1-dimensional Hausdorff measure on it, and let $m$ be the pullback measure on $\partial \mathbb{D}$. By the properties of rectifiable curves, $\xi^1(\partial \mathbb{D})$ has a tangent line $m^*$ almost everywhere. For $m$ a.e. $x\in \partial \mathbb{D}$, we denote the tangent line at $\xi^1(x)$ as $\ell(x)$, which can be identified with a projective line \rev{in} $\mathbb{RP}^2$.

\rev{First we} show that $\ell(x)$ coincides with the continuous dual limit map $\xi^2(x)$ wherever $\ell(x)$ is defined. \rev{Fix any such $x$ and $b_0\in \mathbb{D}$, and let $(\gamma_n) \subset \Gamma_0$ be a sequence with $\gamma_n b_0$ within distance $r_0$ of $[b_0,x)$ such that $\gamma_n b_0\to x$, $\gamma_n^{-1} b_0\to a$, and $\gamma_n^{-1}x \to x' \neq a$ in $\partial\mathbb{D}$}. As $\rho$ is irreducible, there exist \rev{$b,c$ in the two different open arcs of $\partial \mathbb{D}\setminus\{a,x'\}$} such that $\xi^1(a)\oplus \xi^1(b)\oplus \xi^1(c)$ spans $\mathbb{R}^3$. Let $\ell_{b,c}$ be the projective line $\xi^1(b)\oplus \xi^1(c)$.

As $\ell_{b,c}$ is transverse to $\xi^1(a)$, and $\xi^1(b), \xi^1(c)$ are both transverse to $\xi^2(a)$, the strongly dynamics preserving property Lemma~\ref{lemmaofcartanproperty} implies that
\[
\rho(\gamma_n)\ell_{b,c} \to \xi^2(x),\quad \rho(\gamma_n)\xi^1(b)\to \xi^1(x),\quad \text{and} \quad \rho(\gamma_n)\xi^1(c)\to \xi^1(x).
\]
However, $\rho(\gamma_n)\ell_{b,c} = \mathrm{Span}(\rho(\gamma_n)\xi^1(b),\rho(\gamma_n)\xi^1(c))$\rev{, where $\gamma_n(b)$ and $\gamma_n(c)$ approach $x$ from opposite sides (since $\gamma_n^{-1}x \to x'$ separates $b$ and $c$ away from $a$)}. As $n \to \infty$, this sequence of projective lines (spanned by points on the curve converging to $\xi^1(x)$ \rev{from different sides}) converges to the tangent line $\ell(x)$. Thus, we see that $\ell(x) = \xi^2(x)$.

\begin{revblock}
Next, we show that $\xi^1(\partial\mathbb{D})$ is a $C^1$ curve. Let $\sigma$ be its arc-length parametrization, and choose an open arc $J_0 \subset \partial\mathbb{D}$ on which the continuous line field $\xi^2$ is spanned by a continuous unit tangent vector field $v$. Since $\ell = \xi^2$ $m$-a.e., the unit tangent vector $\sigma'$ satisfies $\sigma' = \pm v$ a.e.\ on $J_0$. To deduce that $\xi^1(\partial\mathbb{D})$ is locally $C^1$, it suffices to show that this sign is constant a.e.\ on some non-empty open arc. Replacing $v$ by $-v$ if necessary, the set $E_+ := \{x\in J_0 \mid \sigma' = +v\}$ has $m(E_+) > 0$.

Pick an $m$-Lebesgue density point $x_0 \in E_+$, and choose $(\gamma_n), a, x' \neq a$ as above for $x_0$. Let $J \subset \partial\mathbb{D}\setminus\{a\}$ be a compact arc containing $x'$ in its interior on which $\xi^2$ is spanned by a continuous unit tangent vector field $u$. By transversality of $\xi(J)$ to $\xi(a)$, there exists a constant $C_0 > 1$ such that for all sufficiently large $n$ and every $x \in J$ with unit representative $\hat{p} \in \xi^1(x)$, the tangential stretch factor $\|\mathcal{D}\rho(\gamma_n)(u(x))\| = \frac{\|\rho(\gamma_n)\hat{p} \wedge \rho(\gamma_n) u(x)\|}{\|\rho(\gamma_n)\hat{p}\|^2}$ satisfies
\[
\frac{1}{C_0}\frac{\sigma_2}{\sigma_1}(\rho(\gamma_n)) \le \|\mathcal{D}\rho(\gamma_n)(u(x))\| \le C_0 \frac{\sigma_2}{\sigma_1}(\rho(\gamma_n)).
\]
Since $\sigma' \in \xi^2$ a.e., this stretch factor equals $\frac{d(\gamma_n^* m)}{dm}$ a.e.\ on $J$. Integrating over $J$ and over any measurable subset $A \subset J$ gives
\[
\frac{m(A)}{m(J)} \le C_0^2 \frac{m(\gamma_n(A))}{m(\gamma_n(J))}.
\]

For all sufficiently large $n$, the arcs $\gamma_n(J) \subset J_0$ shrink to $x_0$ with $x_0 \in \gamma_n(J)$, and hence $m(\gamma_n(J)\setminus E_+)/m(\gamma_n(J)) \to 0$. Passing to a subsequence, the direction of $\mathcal{D}\rho(\gamma_n)(u)$ relative to $v$ on the connected arc $J$ has a constant sign $\eta \in \{\pm 1\}$. Let $A := \{x \in J \mid \sigma' \neq \eta u\}$. Since $\gamma_n$ preserves orientation on $\partial\mathbb{D}$, we have $\gamma_n(A) = \gamma_n(J) \setminus E_+$ up to an $m$-null set, so the inequality above yields
\[
\frac{m(A)}{m(J)} \le C_0^2 \frac{m(\gamma_n(J)\setminus E_+)}{m(\gamma_n(J))} \to 0.
\]
Hence $m(A) = 0$, i.e., $\sigma' = \eta u$ $m$-a.e.\ on $J$. Since $\sigma$ is absolutely continuous and its derivative coincides a.e.\ with the continuous vector field $\eta u$ on the arc corresponding to $\mathrm{int}(J)$, the image $\xi^1(\mathrm{int}(J))$ is a $C^1$ arc. Finally, by minimality of $\Gamma_0$ on $\partial\mathbb{D}$ and smoothness of the $\rho(\Gamma_0)$-action, $\xi^1(\partial\mathbb{D})$ is a $C^1$ curve.
\end{revblock}

\textbf{(5) $\implies$ (6):} As shown in the proof of the previous implication, for any $x\in \partial \mathbb{D}$, $\xi^2(x)$ coincides with the tangent line of $\xi^1(\partial \mathbb{D})$ at $\xi^1(x)$. Since the limit maps of Anosov representations are Hölder continuous (see, for example, Canary--Zhang--Zimmer~\cite[Theorem~1.3]{CaZhZi2}), $\xi^1(\partial \mathbb{D})$ is thus a $C^{1,\alpha}$ curve for some $\alpha>0$. As $\rho$ is irreducible, by Zhang--Zimmer~\cite[Theorem~1.9]{ZhangZimmerRegularity}, $\xi^1$ satisfies the following hypertransverse property: \emph{For any distinct $x,y,z\in \partial \mathbb{D}$, $\xi^1(x)\oplus \xi^1(y)\oplus\xi^1(z) = \mathbb{R}^3$.} By Guichard~\cite[Théorème 1]{Guichardcharacterization}, this hypertransverse property implies that $\rho$ is Hitchin.

\textbf{(6) $\implies$ (1):} The limit curves of Hitchin representations are known to be $C^1$ (Labourie~\cite[Theorem~1.4]{labourie2005anosovflowssurfacegroups}). Therefore, the Hausdorff dimension is $1$.
\end{proof}

\section{Quasi-Self-Similarity}\label{sectionofregulardistortion}

\subsection{The Non-folding Property}
Let $\Gamma_0$ be a regular convex cocompact subgroup with respect to $(V,\Omega,\mathcal{C})$ and let $\rho:\Gamma_0\to \mathsf{PGL}(d,\mathbb{R})$ be a projective Anosov representation with limit map $\xi$.

\begin{definition}\label{definitionofnonfolding}
 We say $\rho$ has the \emph{non-folding property} if the following holds:

Fix $b_0\in \mathcal{C}$. For any $R>0,r\ge0$, there exists a constant $C>1$ such that for any $x\in \Lambda^1(\Gamma_0)$, $z\in \mathcal{C}$ and $\gamma\in \Gamma_0$ satisfying that \rev{$z\in [b_0,x)$} and $d_{\Omega}(z,\gamma b_0)\le r$,
we have
    \[B_{d_{\angle}}\left(\xi^{1}(x),\frac{1}{C} e^{-\alpha_1(\kappa(\rho(\gamma)))}\right)\cap \xi^1(\Lambda^1(\Gamma_0))\subset \xi^1(\hat{\mathcal{O}}_{R}(b_0,z))\subset B_{d_{\angle}}\left(\xi^{1}(x),C e^{-\alpha_1(\kappa(\rho(\gamma)))}\right). \]

And we say a subgroup $\Gamma \subset \mathsf{PGL}(d,\mathbb{R})$ is a \emph{projective non-folding Anosov subgroup} if it can be realized as the image $\rho(\Gamma_0)$ as described above.
\end{definition}

There is a criterion for the non-folding property: For any $\gamma \in \Gamma_0$, choose $\omega_{\gamma} \in \Lambda^1(\Gamma_0)$ such that the distance $d_{\angle}(\gamma^{-1} b_0, \omega_{\gamma})$ between $\gamma^{-1} b_0$ and $\omega_{\gamma}$ equals the distance $d_{\angle}(\gamma^{-1} b_0, \Lambda^1(\Gamma_0))$ from $\gamma^{-1} b_0$ to $\Lambda^1(\Gamma_0)$. Similarly, let $\alpha_{\gamma} \in \Lambda^1(\Gamma_0)$ be a point that realizes the distance $d_{\angle}$ between $\gamma b_0$ and $\Lambda^1(\Gamma_0)$. We refer to $\omega_{\gamma}$ and $\alpha_{\gamma}$ as the \emph{coarse repeller} and the \emph{coarse attractor}, respectively. It is worth noting that they are not unique and we are making choices.

\begin{definition}\label{definitionofregulardistortion}
Let $\Gamma_0$ be a regular convex cocompact subgroup with respect to $(V, \Omega ,\mathcal{C})$ and let $\rho:\Gamma_0\to \mathsf{PGL}(d,\mathbb{R})$ be a projective Anosov representation with limit map $\xi$. Fix $b_0\in \mathcal{C}$.   

\begin{enumerate}
\item Let $(x,y,z)$ be a triple in $\overline{\mathcal{C}}\cap \partial \Omega$. Given any $d>0$, we say $(x,y,z)$ is \emph{$d$-bounded} if $d_{\angle}(x,y)$, $d_{\angle}(x,z)$, $d_{\angle}(y,z)\ge d$.
\item $\rho$ has the \emph{regular distortion property} if for any $d>0$, there is a constant $C>1$ depending only on $d$, such that for any $p,q\in \Lambda^1(\Gamma_0)$ and $\gamma \in \Gamma_0$ where $(\omega_{\gamma},p,q)$ is $d$-bounded, then
$$
C e^{-\alpha_1(\kappa(\rho(\gamma)))}\ge d_{\angle}(\rho(\gamma)\xi^{1}(p),\rho(\gamma)\xi^{1}(q))\ge \frac{1}{C} e^{-\alpha_1(\kappa(\rho(\gamma)))}
$$
where $\xi^{1}$ is the limit map from $\Lambda^1(\Gamma_0)$ to $\Lambda^1(\Gamma)$.
\end{enumerate}
\end{definition}

The regular distortion property implies the following shadow radius estimate, whose proof is postponed until Appendix~\ref{append1}.

\begin{proposition}[{Canary-Zhang-Zimmer~\cite[Theorem~5.1 and Theorem~7.1]{CaZhZi}}]\label{shadowradiusstrong}
If $\rho$ satisfy the regular distortion property, then $\rho$ has the non-folding property. 
\end{proposition}

\begin{remark}\label{listofexamples}
Numerous examples of projective Anosov representations satisfy these assumptions, and for most of them, the Hausdorff dimension of the limit set is computable (see the references):
\begin{enumerate}
    \item When $\Gamma_0$ is a convex cocompact subgroup in $\mathsf{SO}(n,1)$ where $n>1$, the identity representation satisfies the non-folding property (Sullivan~\cite{sullivanoriginal}).
    \item When $\rho$ is a $(1,1,2)$-hyperconvex representation, it satisfies the regular distortion property (Pozzetti--Sambarino--Wienhard~\cite{pozzetti2020conformalityrobustclassnonconformal}, see also Canary--Zhang--Zimmer~\cite{CaZhZi}).
    \item When $\Gamma_0$ is a convex cocompact subgroup acting on $\mathbb{D}$, and $\rho$ is the composition of a $\Theta$-positive representation and an adapted representation with respect to some $\alpha\in \Theta$ (see Yao~\cite[Lemma~2.2 and Proposition~2.16]{yaothetapositive}), then $\rho$ satisfies the regular distortion property (\cite[Theorem~4.2]{yaothetapositive}).    
\end{enumerate}
\end{remark}

\subsection{Quasi-Self-Similarity}
We follow Falconer~\cite{falconerselfsimilar1} to introduce the concept of quasi-self-similarity.

\begin{definition}\label{definitionofquasiselfsimilarity}
A non-empty compact metric space $(F,d)$ is called \emph{expansively quasi-self-similar} if there exist constants $\lambda_0,a_0>0$ such that for any set $N\subset F$ with $\mathrm{diam}(N)<\lambda_0$, there exists a mapping $\phi:N\to F$ such that
\[
a_0 d(x,y)\le \mathrm{diam}(N) d(\phi(x),\phi(y)) \quad \text{for all } x,y\in N.
\]
\end{definition}

And we introduce the lower and upper box-counting dimensions.

\begin{definition}\label{definitionofboxcountingdimension}
Let $(F,d)$ be a metric space, and for $\epsilon>0$, let $n(\epsilon)$ be the largest number of disjoint balls of radius $\epsilon$ centered in $F$. The \emph{lower box-counting dimension} and the \emph{upper box-counting dimension} are defined by
\[
\underline{\dim}_B F = \liminf_{\epsilon\to 0}  - \frac{\log n(\epsilon)}{\log \epsilon},
\]
and
\[
\overline{\dim}_B F = \limsup_{\epsilon\to 0}  - \frac{\log n(\epsilon)}{\log \epsilon}.
\]
\end{definition}

\begin{theorem}[{Falconer~\cite[Theorem~3]{falconerselfsimilar1}}]\label{theoremoffalconerdimension}
If $(F,d)$ is expansively quasi-self-similar, then
\[
\dim_H F = \underline{\dim}_B F = \overline{\dim}_B F.
\]
\end{theorem}

\subsection{Hausdorff Dimension of the Non-Folding Limit Sets}
In this Section, we always work with the underlying angular metric $d_{\angle}$. We prove the following theorem:

\begin{theorem}\label{theoremofnonfoldingdimension}
Let $\Gamma_0$ be a regular convex cocompact subgroup with respect to $(V,\Omega,\mathcal{C})$, and let $\rho:\Gamma_0\to \mathsf{PGL}(d,\mathbb{R})$ be a projective Anosov representation with limit map $\xi$. If $\rho$ has the non-folding property, then 
\[
\dim_H \big(\xi^1 (\Lambda^1(\Gamma_0))\big) = \underline{\dim}_B \big(\xi^1 (\Lambda^1(\Gamma_0))\big) = \overline{\dim}_B \big(\xi^1 (\Lambda^1(\Gamma_0))\big) = \delta^{\alpha_1}(\rho(\Gamma_0)),
\]
where $\delta^{\alpha_1}(\rho(\Gamma_0))$ is the critical exponent defined by the convergence abscissa of the associated Poincare series:
\[
\mathcal{Q}_{\rho(\Gamma_0)}^{\alpha_1}(s) = \sum_{\gamma \in \Gamma_0} e^{-s\, \alpha_1(\kappa(\rho(\gamma)))}.
\]

\end{theorem}

\vspace{3mm}

The theorem follows from the following two propositions. Let $\Gamma_0, V, \Omega, \mathcal{C}, \rho$, and $\xi$ be as described in Theorem~\ref{theoremofnonfoldingdimension}.

\begin{proposition}\label{propositionoftheboxcountinglowerbound}
$\overline{\dim}_B \big(\xi^1(\Lambda^1(\Gamma_0))\big) \ge \delta^{\alpha_1}(\rho(\Gamma_0)).$
\end{proposition}

\begin{proof}
Fix $b_0\in \mathcal{C}$. We can pick $R>0$ sufficiently large such that for any $\gamma\in \Gamma_0$, the shadow $\hat{\mathcal{O}}_{\frac{R}{2}}(b_0,\gamma b_0)$ is non-empty (for example, by Theorem~\ref{shadowlemma}, for any Patterson--Sullivan measure, any shadow of sufficiently large radius has positive measure, and is therefore non-empty). Moreover, for any $x \in \hat{\mathcal{O}}_{\frac{R}{2}}(b_0,\gamma b_0)$, there exists a point $z \in [b_0,x)$ such that $d_{\Omega}(z,\gamma b_0)\le \frac{R}{2}$ and $\hat{\mathcal{O}}_{\frac{R}{2}}(b_0,z)\subset \hat{\mathcal{O}}_R(b_0,\gamma b_0)$. For each $\gamma \in \Gamma_0$, we then fix a point $x(\gamma) \in \hat{\mathcal{O}}_{\frac{R}{2}}(b_0,\gamma b_0)$ and let $z(\gamma) \in [b_0,x(\gamma))$ be the corresponding point satisfying these properties.

Let $C_0 = C_0(R)$ be the constant given in Lemma~\ref{lemmaofdisjointshadows}. As $\Gamma_0$ is discrete, there exists a finite partition $\Gamma_0 = \Gamma_1\cup\dots\cup \Gamma_k$ such that for any $1\le i\le k$ and $n\ge 0$, the set $(\Gamma_i\cap  \mathcal{A}_n(\alpha_1))\cdot  b_0$ is $C_0$-separated under $d_{\Omega}$.

By Lemma~\ref{lemmaofdivergence}, there exists $1\le i_0\le k$ such that for any $0<\delta<\delta^{\alpha_1}(\rho(\Gamma_0))$,
\[
\limsup_{n\to \infty} \sum_{\gamma\in \mathcal{A}_n(\alpha_1)\cap \Gamma_{i_0}} e^{-\delta\alpha_1(\kappa(\rho(\gamma)))} = \infty.
\]
As $\rho$ satisfies the non-folding property, there exists a constant $C>1$ such that for any $\gamma\in \Gamma_0$,
\[
B_{d_{\angle}}\left(\xi^1(x(\gamma)),\frac{1}{C} e^{-\alpha_1(\kappa(\rho(\gamma)))}\right)\cap \xi^1(\Lambda^1(\Gamma_0))\subset \xi^1\big(\hat{\mathcal{O}}_{\frac{R}{2}}(b_0,z(\gamma))\big)\subset \xi^1(\hat{\mathcal{O}}_{R}(b_0,\gamma b_0)).
\]

For $n\gg 1$, by Lemma~\ref{lemmaofdisjointshadows}, the collection of balls 
\[
\left\{B_{d_{\angle}}\left(\xi^1(x(\gamma)),\frac{1}{C} e^{-\alpha_1(\kappa(\rho(\gamma)))}\right) \cap \xi^1(\Lambda^1(\Gamma_0)) \;\middle|\; \gamma \in \mathcal{A}_n(\alpha_1)\cap \Gamma_{i_0}\right\}
\]
is pairwise disjoint (since the shadows are disjoint). By the construction of $\mathcal{A}_n(\alpha_1)$, the radii of the balls in this family are at least $\frac{e^{-(n+1)}}{C}$. Therefore, $n\left(\frac{e^{-(n+1)}}{C}\right)\ge |\mathcal{A}_n(\alpha_1)\cap \Gamma_{i_0}|$.

Lemma~\ref{lemmaofdivergence} shows that 
\[
\limsup_{n\to \infty } e^{-n\delta} |\mathcal{A}_n(\alpha_1)\cap \Gamma_{i_0}| = \infty.
\]
This implies
\[
\limsup_{n\to \infty } \big(\log |\mathcal{A}_n(\alpha_1)\cap \Gamma_{i_0}|-n\delta\big) = \infty,
\]
and consequently,
\[
\limsup_{n\to \infty } \left(\log n\left(\frac{e^{-(n+1)}}{C}\right)-n\delta\right) = \infty.
\]
Let $\epsilon_n = \frac{e^{-(n+1)}}{C}$. Then $-n = \log \epsilon_n + \log C + 1$, so
\[
\limsup_{n\to \infty } \big(\log n(\epsilon_n)+(\log \epsilon_n+\log C + 1)\delta\big) = \infty.
\]
Since $(\log C + 1)\delta$ is a constant, this simplifies to
\[
\limsup_{n\to \infty } \big(\log n(\epsilon_n)+ \delta\log \epsilon_n\big) = \infty.
\]
Dividing by $-\log \epsilon_n$ (which is positive and tends to $\infty$ as $n\to\infty$), we deduce
\[
\limsup_{n\to \infty} \left(-\frac{\log n(\epsilon_n)}{\log \epsilon_n }\right)-\delta \ge 0.
\]
Thus,
\[
\overline{\dim}_B \big(\xi^1(\Lambda^1(\Gamma_0))\big) \ge \delta.
\]
Since $\delta < \delta^{\alpha_1}(\rho(\Gamma_0))$ is arbitrary, this concludes the proof.
\end{proof}

\begin{proposition}\label{propositionofquasiselfsimilarityofthelimitset}
The metric space $\big(\xi^1(\Lambda^1(\Gamma_0)), d_{\angle}\big)$ is expansively quasi-self-similar.
\end{proposition}

\begin{proof}
    Fix $b_0\in \mathcal{C}$. Similar to the arguments in the proof of Proposition~\ref{propositionoftheboxcountinglowerbound}, there exist a sufficiently large $R>0$ and a constant $C>1$ such that for any $\gamma\in \Gamma_0$, $\hat{\mathcal{O}}_{\frac{R}{2}}(b_0,\gamma b_0)$ is non-empty, and we have for any $x\in \hat{\mathcal{O}}_{\frac{R}{2}}(b_0,\gamma b_0)$,
    \begin{equation}\label{eqofexpansiveselfsimilar}
           B_{d_{\angle}}\left(\xi^1(x),\frac{1}{C} e^{-\alpha_1(\kappa(\rho(\gamma)))}\right)\cap \xi^1(\Lambda^1(\Gamma_0))\subset \xi^1(\hat{\mathcal{O}}_{R}(b_0,\gamma b_0)).
    \end{equation}
 
    Furthermore, by Lemma~\ref{lemmaofellipseandshadows}, we can enlarge $R$ and $C$ if necessary to assume that for all but finitely many $\gamma\in \Gamma_0$,
    \begin{equation}\label{eqofexpansiveselfsimilar2}
         \mathcal{E}\left(\rho(\gamma),\frac{1}{C}\right)\cap \Lambda^1(\rho(\Gamma_0))\subset \xi^1( \hat{\mathcal{O}}_R(b_0,\gamma b_0))\subset \mathcal{E}(\rho(\gamma), C).   
    \end{equation}

    We further enlarge $R$ and $C$ if necessary to ensure that $\frac{R}{2}$ satisfies the requirement of Lemma~\ref{lemmaofquasigeodesics}, and let $a>1, A>0$ be the constants therein. 

    Let $F\subset \Gamma_0$ be the finite subset where Equation~\ref{eqofexpansiveselfsimilar2} fails. From the definition of $\mathcal{E}(\rho(\gamma), C)$, it is straightforward to verify the following:

    \begin{observation}
    There exists a constant $C'$ such that for any $\gamma\in \Gamma_0 \setminus F$ and any $p_1,p_2\in \mathcal{E}(\rho(\gamma), C)$,
    \[
    d_{\angle}(\rho(\gamma)^{-1} p_1, \rho(\gamma)^{-1} p_2) \ge C' \frac{\sigma_1}{\sigma_2}(\rho(\gamma)) d_{\angle}(p_1,p_2).
    \]
    \end{observation}

\vspace{3mm}

    Pick $x\in \Lambda^1(\Gamma_0)$. Let $(\gamma_n)_{n\in \mathbb{Z}^{>0}}$ be the sequence representing an $(a,A)$-quasi-geodesic ray as described in Lemma~\ref{lemmaofquasigeodesics}. Then by the assumption on $R$, we have for any $n>0$,
    \[
    x\in \hat{\mathcal{O}}_{\frac{R}{2}}(b_0,\gamma_n b_0),
    \]
    and for any $\gamma_n\in \Gamma_0\setminus F$,
    \[
    B_{d_{\angle}}\left(\xi^1(x),\frac{1}{C} e^{-\alpha_1(\kappa(\rho(\gamma_n)))}\right)\cap \xi^1(\Lambda^1(\Gamma_0))\subset \xi^1(\hat{\mathcal{O}}_{R}(b_0,\gamma_n b_0))\subset \mathcal{E}(\rho(\gamma_n), C).
    \]

    From Lemma~\ref{lemmaofuniformcontinuity} and the $(a,A)$-quasi-geodesic property of $(\gamma_n)$ starting from $\gamma_0 = \mathrm{id}$, there exists a constant $C''>1$ independent of $x$ such that
    \begin{equation}\label{eqofquasiselfsimilar3}
    \frac{1}{C''} e^{-\alpha_1(\kappa (\rho(\gamma_{n+1})))}\le e^{-\alpha_1(\kappa (\rho(\gamma_{n})))}\le C'' e^{-\alpha_1(\kappa (\rho(\gamma_{n+1})))}.
    \end{equation}
    Also note that due to the $(a,A)$-quasi-geodesic property in Lemma~\ref{lemmaofquasigeodesics}, there exist constants $N_0,\epsilon_0>0$ independent of $x$ and the choice of $(\gamma_n)$ such that
    \begin{enumerate}
        \item If $n>N_0$, then $\gamma_n\in \Gamma_0\setminus F$.
        \item $\frac{\sigma_2}{\sigma_1}(\rho(\gamma_{N_0+1}))>\epsilon_0$. 
    \end{enumerate}

\vspace{3mm}

    Let $a_0 = \frac{C'}{C'' C} $ and $\lambda_0 = \frac{\epsilon_0}{CC''} $. Suppose $N\subset \xi^1(\Lambda^1(\Gamma_0))$ and $\mathrm{diam}(N)<\lambda_0$. Pick $p\in N$ and $x\in \Lambda^1(\Gamma_0)$ such that $p = \xi^1(x)$, and pick $(\gamma_n)_{n\in \mathbb{Z}^{>0}}$ representing a quasi-geodesic ray as mentioned above. By Equation~\ref{eqofquasiselfsimilar3} and $\lim_{n\to \infty} \frac{\sigma_2}{\sigma_1}(\rho(\gamma_n)) = 0$, there exists $n'>N_0$ such that
    \begin{equation}\label{equationofbootstrap}
    C \mathrm{diam}(N)\le e^{-\alpha_1(\kappa(\rho(\gamma_{n'})))}  \le  C'' (C \mathrm{diam}(N)).
    \end{equation}
    As a result,
    \[
    N\subset B_{d_{\angle}}\left(p,\frac{1}{C} e^{-\alpha_1(\kappa(\rho(\gamma_{n'})))}\right)\cap \xi^1(\Lambda^1(\Gamma_0))\subset \xi^1(\hat{\mathcal{O}}_{R}(b_0,\gamma_{n'} b_0))\subset \mathcal{E}(\rho(\gamma_{n'}), C).
    \]

    By the Observation above, we have for any $p_1,p_2\in N$,
    \[
    d_{\angle}(\rho(\gamma_{n'})^{-1}p_1, \rho(\gamma_{n'})^{-1} p_2) \ge C' \frac{\sigma_1}{\sigma_2}(\rho(\gamma_{n'})) d_{\angle}(p_1,p_2),
    \]
    which implies that 
    \[
    e^{-\alpha_1(\kappa(\rho(\gamma_{n'})))} d_{\angle}(\rho(\gamma_{n'})^{-1}p_1, \rho(\gamma_{n'})^{-1} p_2) \ge C'  d_{\angle}(p_1,p_2).
    \]
    Thus, by Equation~\ref{equationofbootstrap}
    \[
    C'' C (\mathrm{diam}(N)) d_{\angle}(\rho(\gamma_{n'})^{-1}p_1, \rho(\gamma_{n'})^{-1} p_2) \ge C'  d_{\angle}(p_1,p_2),
    \]
    so
    \[
    (\mathrm{diam}(N)) d_{\angle}(\rho(\gamma_{n'})^{-1}p_1, \rho(\gamma_{n'})^{-1} p_2) \ge a_0  d_{\angle}(p_1,p_2),
    \]
    which concludes the proof.
\end{proof}

\begin{proof}[Proof of Theorem~\ref{theoremofnonfoldingdimension}]

Theorem~\ref{theoremoffalconerdimension} and Proposition~\ref{propositionoftheboxcountinglowerbound} show that 
\[
\dim_H \big(\xi^1 (\Lambda^1(\Gamma_0))\big) = \underline{\dim}_B \big(\xi^1 (\Lambda^1(\Gamma_0))\big) = \overline{\dim}_B \big(\xi^1 (\Lambda^1(\Gamma_0))\big) \ge \delta^{\alpha_1}(\rho(\Gamma_0)).
\]
On the other hand, Canary--Zhang--Zimmer~\cite[Corollary~5.2]{CaZhZi} shows that $\dim_H \big(\xi^1 (\Lambda^1(\Gamma_0))\big)\le \delta^{\alpha_1}(\rho(\Gamma_0))$, which concludes the proof.
\end{proof}

\appendix

\section{Proof of Proposition \ref{shadowradiusstrong}}\label{append1}

The proof follows Canary--Zhang--Zimmer~\cite[Theorem~5.1 and Theorem~7.1]{CaZhZi}.

Let $V$ be a real vector space and let $\Omega \subset \mathbb{P}(V)$ be a proper convex domain. For any $a,b\in \mathbb{P}(V)$ with $a\neq b$, let $l_{ab}$ denote the projective line through $a$ and $b$ (which, when $a,b\in \overline{\Omega}$, $l_{ab}\cap \Omega$ restricts to the geodesic in $\Omega$ joining them). 
If $x,y\in \overline{\Omega}$, we write $[x,y]$, $(x,y)$, $[x,y)$, and $(x,y]$ for the closed, open, and the two half-open geodesic segments with endpoints $x$ and $y$.

Assume $\mathcal{C}\subset \Omega$ is a closed convex set such that every point in $\overline{\mathcal{C}}\cap \partial \Omega$ is $C^1$ and extreme. Fix $b_0\in \mathcal{C}$. For $x\in \partial\Omega\cap \overline{\mathcal{C}}$, let $x^{opp}$ be the second intersection point of $l_{b_0x}$ with $\partial\Omega$. Let $\mathring{x}$ be the codimension-2 projective subspace obtained as the intersection of the two \rev{chosen} supporting hyperplanes to $\Omega$ at $x$ and $x^{opp}$\rev{, where the supporting hyperplane at $x$ is indeed unique}. Define $\pi_x:\partial\Omega\to [x,x^{opp}]$ by
\[
\pi_x(y)\;=\; \mathrm{Span}(\mathring{x},y)\cap [x,x^{opp}].
\]

From now on, we assume that $\Gamma_0$ is a regular convex cocompact subgroup with respect to $(V,\Omega,\mathcal{C})$, and fix $b_0\in \mathcal{C}$.

We invoke some lemmas from Canary--Zhang--Zimmer~\cite{CaZhZi}.

\begin{lemma}[{\cite[Lemma~7.5]{CaZhZi}}]\label{lem9}
For any $R>0$, there exists $T>0$ with the following property: if $x,y\in \Lambda^1(\Gamma_0)$, $z\in [b_0,x)$, $\pi_x(y)\in (z,x)$, and $d_{\Omega}\big(\pi_x(y),z\big)\ge T$, then $y\in \hat{\mathcal{O}}_R(b_0,z)$.
\end{lemma}

Note that the variable $r'$ and the requirement $d_{\Omega}\big(z,\Gamma_0(b_0)\big)\le r'$ in the reference are redundant in our setting because $\Gamma_0$ acts cocompactly on $\mathcal{C}$.

\begin{lemma}[{\cite[Lemma~7.6]{CaZhZi}}]\label{lem10}
Let $\rho:\Gamma_0\to \mathsf{PGL}(d,\mathbb{R})$ be a projective Anosov representation with limit map $\xi$ that satisfies the regular distortion property. For any $r\ge 0$ and $0<\delta<1$, there exists a constant $C_1>1$ such that for any $x,y\in \Lambda^1(\Gamma_0)$ and $\gamma\in \Gamma_0$, if
\[
d_{\Omega}\big(\gamma b_0,\pi_x(y)\big)\le r
\quad\text{and}\quad
d_{\angle}(x,y)\le \delta,
\]
then
\[
d_{\angle}\big(\xi^1(x),\xi^1(y)\big) \ge \frac{1}{C_1} e^{-\alpha_1(\kappa(\rho(\gamma)))}.
\]
\end{lemma}

\begin{proof}
Assume for contradiction that the claim fails. Then there exist sequences $x_n,y_n\in \Lambda^1(\Gamma_0)$ and $\gamma_n\in \Gamma_0$ such that
\[
0 < d_{\angle}(x_n,y_n)\le \delta,\qquad
d_{\Omega}\big(\gamma_n b_0,\pi_{x_n}(y_n)\big)\le r,
\]
but
\begin{equation}\label{inequalityinappend1}
   d_{\angle}\big(\xi^1(x_n),\xi^1(y_n)\big) \le \frac{1}{n} e^{-\alpha_1(\kappa(\rho(\gamma_n)))}. 
\end{equation}
Set $p_n=\gamma_n^{-1} y_n$ and $q_n=\gamma_n^{-1} x_n$. Passing to a subsequence, assume
\[
x_n\to x\in \Lambda^1(\Gamma_0),\quad y_n\to y\in \Lambda^1(\Gamma_0),\qquad
\]
\[
\gamma_n b_0\to b\in \overline{\Omega},\quad \gamma_n^{-1} b_0\to a\in \overline{\Omega},\quad
\gamma_n^{-1}\big(\pi_{x_n}(y_n)\big)\to z\in \overline{B_{d_{\Omega}}(b_0,r)},
\]
and
\[
p_n\to p\in \Lambda^1(\Gamma_0),\qquad q_n\to q\in \Lambda^1(\Gamma_0),\qquad \gamma_n^{-1} x_n^{opp} \to \hat{a} \in \Lambda^1(\Gamma_0).
\]

Since the right-hand side of Inequality~\ref{inequalityinappend1} above tends to $0$ as $n\to\infty$, we have
$d_{\angle}\big(\xi^1(x_n),\xi^1(y_n)\big)\to 0$, and it follows that $x=y$. Hence $d_{\angle}(x_n,y_n)\to 0$ and $\pi_{x_n}(y_n)\to x$. As $d_{\Omega}\big(\gamma_n b_0,\pi_{x_n}(y_n)\big)\le r$, we obtain $\gamma_n b_0\to x$. Thus, the sequence $(\gamma_n)$ is unbounded, which implies $a,b\in \Lambda^1(\Gamma_0)$; moreover, $\hat{a}=a$.

Let $l_n := \gamma_n^{-1} l_{b_0 x_n}$. Then $\gamma_n^{-1} b_0, \gamma_n^{-1}\big(\pi_{x_n}(y_n)\big)\in l_n$. Since $\gamma_n^{-1}\big(\pi_{x_n}(y_n)\big)\to z\in \overline{B_{d_{\Omega}}(b_0,r)}$, the points $a, p, q$ are pairwise distinct. Note also that $\omega_{\gamma_n}\to a$. Hence, there exists $d>0$ such that, for all sufficiently large $n$, the triple $(\omega_{\gamma_n}, p_n, q_n)$ is $d$-bounded. By the regular distortion property, there exists $C_2>0$ such that
\[
d_{\angle}\big(\xi^1(x_n),\xi^1(y_n)\big)
= d_{\angle}\big(\rho(\gamma_n)\xi^1(q_n),\,\rho(\gamma_n)\xi^1(p_n)\big)
\ge C_2\, e^{-\alpha_1(\kappa(\rho(\gamma_n)))},
\]
which contradicts $d_{\angle}\big(\xi^1(x_n),\xi^1(y_n)\big)\le \frac{1}{n} e^{-\alpha_1(\kappa(\rho(\gamma_n)))}$. 
\end{proof}

\begin{proof}[Proof of Proposition~\ref{shadowradiusstrong}]

The right half of the inclusion is due to Canary--Zhang--Zimmer~\cite[Theorem~5.1]{CaZhZi}. So it suffices to show that for any $R>0, r\ge 0$ there exists a constant $C>1$ such that for any $x\in \Lambda^1(\Gamma_0)$, $z\in [b_0,x)$, and $\gamma\in \Gamma_0$ with $d_{\Omega}(z,\gamma b_0)\le r$, one has
\[
B_{d_{\angle}}\left(\xi^1(x),\frac{1}{C}\,e^{-\alpha_1(\kappa(\rho(\gamma)))}\right)\cap \xi^1(\Lambda^1(\Gamma_0))
\subset \xi^1 (\hat{\mathcal{O}}_{R}(b_0,z)).
\]

Assume not. Then there exist sequences $x_n,y_n\in \Lambda^1(\Gamma_0)$, $z_n\in [b_0,x_n)$, and $\gamma_n\in \Gamma_0$ such that
\[
d_{\Omega}(z_n,\gamma_n b_0)\le r,\qquad y_n\notin \hat{\mathcal{O}}_{R}(b_0,z_n),\qquad
d_{\angle}\big(\xi^1(x_n),\xi^1(y_n)\big)<\frac{1}{n}\,e^{-\alpha_1(\kappa(\rho(\gamma_n)))}.
\]
Passing to a subsequence, assume $x_n\to x\in \Lambda^1(\Gamma_0)$, $y_n\to y\in \Lambda^1(\Gamma_0)$, and $\gamma_n b_0\to b\in \overline{\Omega}$. Since $\xi^1$ is continuous and injective and $d_{\angle}\big(\xi^1(x_n),\xi^1(y_n)\big)<\frac{1}{n}\,e^{-\alpha_1(\kappa(\rho(\gamma_n)))}$, we must have $x=y$.

Let $w_n:=\pi_{x_n}(y_n)$. Since $\Gamma_0$ acts cocompactly on $\mathcal{C}$, there exists $\beta_n\in \Gamma_0$ with $d_{\Omega}(w_n,\beta_n b_0)\le D$, where $D$ is the $d_{\Omega}$-diameter of the compact fundamental domain of the $\Gamma_0$-action on $\mathcal{C}$. By Lemma~\ref{lem10}, there exists $C_1>1$ such that
\[
d_{\angle}\big(\xi^1(x_n),\xi^1(y_n)\big) \ge \frac{1}{C_1}\,e^{-\alpha_1(\kappa(\rho(\beta_n)))}.
\]
Together with $d_{\angle}\big(\xi^1(x_n),\xi^1(y_n)\big)<\frac{1}{n}\,e^{-\alpha_1(\kappa(\rho(\gamma_n)))}$, this yields
\[
\alpha_1\big(\kappa(\rho(\beta_n))\big) \ge \alpha_1\big(\kappa(\rho(\gamma_n))\big)+\log n-\log C_1. 
\]

Set $\eta_n:=\beta_n^{-1}\gamma_n\in \Gamma_0$. Lemma~\ref{lemmaofuniformcontinuity} and the above sequence of inequalities implies that $\{\eta_n\}$ is unbounded, hence $d_{\Omega}(b_0,\eta_n b_0)\to\infty$. It follows that $d_{\Omega}(\gamma_n b_0,\beta_n b_0)\to\infty$, and therefore $d_{\Omega}(z_n,w_n)\to\infty$ since $d_{\Omega}(z_n,\gamma_n b_0)\le r$ and $d_{\Omega}(w_n,\beta_n b_0)\le D$.

By Lemma~\ref{lem9}, $w_n$ cannot lie in $(z_n,x_n)$ for infinitely many $n$ (otherwise $y_n$ would be in the shadow), so after passing to a subsequence, we may suppose $w_n\in (b_0,z_n]$. Using $d_{\Omega}(\gamma_n b_0,z_n)\le r$ and $d_{\Omega}(\beta_n b_0,w_n)\le D$, we estimate
\begin{align*}
d_{\Omega}\big(\beta_n b_0,[b_0,\gamma_n b_0]\big)
&\le d_{\Omega}\big(\beta_n b_0,w_n\big)+ d_{\Omega}\big(w_n,[b_0,\gamma_n b_0]\big) \\
&\le D+ d_{\Omega}\big(z_n,[b_0,\gamma_n b_0]\big) \\
&\le D+r.
\end{align*}
Here we used the fact from Hilbert geometry (see Crampon~\cite[Lemma~8.3]{Cramponentropy} and Blayac~\cite[Section~A]{blayacmixing} ) that for any $y\in \overline{\Omega}, z\in \Omega$, $w\in[b_0,z] \implies d_{\Omega}\big(w,[b_0,y)\big)\le  d_{\Omega}\big(z,[b_0,y)\big)$.

Applying \cite[Lemma~6.6]{CaZhZi}, we obtain a constant $C_0>0$ such that
\[
\alpha_1\big(\kappa(\rho(\gamma_n))\big) \ge \alpha_1\big(\kappa(\rho(\beta_n))\big) + \alpha_1\big(\kappa(\rho(\beta_n^{-1}\gamma_n))\big)-C_0 \ge \alpha_1\big(\kappa(\rho(\beta_n))\big)-C_0,
\]
which contradicts the inequality $\alpha_1\big(\kappa(\rho(\beta_n))\big) \ge \alpha_1\big(\kappa(\rho(\gamma_n))\big)+\log n-\log C_1$ for all sufficiently large $n$. This proves the left half inclusion.
\end{proof}

\bibliographystyle{amsplain}
\bibliography{ref}

\end{document}